\documentclass[11pt]{article}
\usepackage[margin=1in]{geometry}
\usepackage{latexsym, amscd, amsfonts, mathrsfs, amsmath, amssymb, amsthm}
\usepackage{bm}
\usepackage{hyperref,color}
\usepackage[noend]{algpseudocode}
\usepackage[capitalize,nameinlink]{cleveref}

\makeatletter
\newcounter{algorithmicH}
\let\oldalgorithmic\algorithmic
\renewcommand{\algorithmic}{%
  \stepcounter{algorithmicH}
  \oldalgorithmic}
\renewcommand{\theHALG@line}{ALG@line.\thealgorithmicH.\arabic{ALG@line}}
\makeatother

\usepackage[dvipsnames]{xcolor}
\hypersetup{
	colorlinks=true,
	pdfpagemode=UseNone,
    citecolor=OliveGreen,
    linkcolor=NavyBlue,
    urlcolor=Magenta,
	pdfstartview=FitW
}
\usepackage{stmaryrd, tikz-cd, enumitem, mathrsfs, bbm, algorithm, url, esint, todonotes, listings, moreverb, makecell}

\def\bE{\mathbb{E}}

\def\bP{\mathbb{P}}
\def\bR{\mathbb{R}}

\def\bZ{\mathbb{Z}}

\def\cE{\mathcal{E}}
\def\cF{\mathcal{F}}
\def\cG{\mathcal{G}}

\def\cP{\mathcal{P}}
\def\cS{\mathcal{S}}
\def\cX{\mathcal{X}}
\def\cY{\mathcal{Y}}

\DeclareMathOperator\Diag{\mathrm{Diag}}

\DeclareMathOperator\Var{\mathrm{Var}}
\DeclareMathOperator\Unif{\mathrm{Unif}}
\DeclareMathOperator\supp{\mathrm{supp}}

\DeclareMathOperator\Ent{\mathrm{Ent}}

\DeclareMathOperator\KL{\mathrm{KL}}
\DeclareMathOperator\TV{\mathrm{TV}}

\DeclareMathOperator\Pois{\mathrm{Pois}}

\newcommand{\nbyp}[1]{{\color{red}[Yury: #1]}}

\newcommand{\rom}[1]{\textup{\uppercase\expandafter{\romannumeral#1}}}

\newtheorem{theorem}{Theorem}
\newtheorem{lemma}[theorem]{Lemma}
\newtheorem{proposition}[theorem]{Proposition}

\newtheorem{claim}[theorem]{Claim}

\theoremstyle{definition}
\newtheorem{definition}[theorem]{Definition}
\newtheorem{example}[theorem]{Example}

\begin{document}
\title{Entropy Contractions in Markov Chains: Half-Step, Full-Step and Continuous-Time}
\author{
  Pietro Caputo\thanks{\texttt{pietro.caputo@uniroma3.it}. Universit\`{a} Roma Tre.}
  \and
  Zongchen Chen\thanks{\texttt{chenzongchen@gatech.edu}. Georgia Institute of Technology.}
  \and
  Yuzhou Gu\thanks{\texttt{yuzhougu@nyu.edu}. New York University.}
  \and
  Yury Polyanskiy\thanks{\texttt{yp@mit.edu}. Massachusetts Institute of Technology.}
}
\date{}

\maketitle

\begin{abstract}
  This paper considers the speed of convergence (mixing) of a finite Markov kernel $P$ with respect to
  the Kullback-Leibler divergence (entropy). Given a Markov kernel one defines either a discrete-time
  Markov chain (with the $n$-step transition kernel given by the matrix power $P^n$) or a
  continuous-time Markov process (with the time-$t$ transition kernel given by
  $e^{t(P-\mathrm{Id})}$). The contraction of entropy for $n=1$ or $t=0+$ are characterized by the famous
  functional inequalities,
  the \textit{strong data processing inequality (SDPI)} and the \textit{modified log-Sobolev inequality
  (MLSI)}, respectively.
  When $P=KK^*$ is written as the product of a kernel and its adjoint, one could also consider the ``half-step'' contraction, which is the SDPI for $K$, while the ``full-step'' contraction refers to the SDPI for $P$.
  The work~\cite{del2003contraction} claimed that these contraction
  coefficients (half-step, full-step, and continuous-time) are generally comparable, that is their ratio is bounded from above and below by
  two absolute constants. We disprove this and related conjectures by working out a number of
  different counterexamples. In particular, we construct (a)
  a continuous-time Markov process that contracts arbitrarily faster than its discrete-time
  counterpart; and
  (b) a kernel $P$ such that $P^{m+1}$ contracts arbitrarily better than $P^m$.
  Hence, our main conclusion is that the four standard inequalities comparing five common notions of entropy and variance contraction are generally not
  improvable (even in the subclass of factorizable Markov chains).

  In the process of analyzing the counterexamples, we survey and sharpen the tools for bounding the
  contraction coefficients and characterize properties of extremizers of
  the respective functional inequalities.
  For example, we show that while the MLSI extremizers
  always have full support (unless the MLSI constant equals twice the spectral gap), the SDPI
  extremizers can have partial support.
  As our examples range from Bernoulli-Laplace model, random walks on
  graphs, to birth-death chains, the paper is also intended as a tutorial on computing MLSI, SDPI
  and other constants for these types of commonly occurring Markov chains.

\end{abstract}

\tableofcontents

\section{Introduction} \label{sec:intro}



Markov chains are widely used in almost all areas of science and engineering, in the form
of MCMC averaging in numerical analysis, approximation algorithms in computer science,
generative modeling in artificial intelligence, and so on. Perhaps the most important problem in the study of
Markov chains is to understand their equilibration properties. One example of such property is
the mixing time, characterizing  the
time it takes for the marginal distribution to come close to the stationary one, as measured by
some statistical distance: the total variation (most commonly) or Wasserstein
distance, $\chi^2$ or Kullback-Leibler (KL) divergence.

In this paper, we focus on Markov chains on a finite state space.  There are two common ways to
define or implement Markov processes.  In a discrete-time Markov chain, the state is updated at
every discrete time $t \in \bZ_{\ge 1}$; meanwhile, in a continuous-time Markov chain, the update
times are distributed as a Poisson point process on the positive real axis $\bR_{\ge 0}$.
For example, if the single-step kernel for the former is given by
a row-stochastic matrix $P$ then the corresponding continuous-time chain has kernel $T_t =
e^{t(P-\mathrm{Id})}$. Due to concentration of the number of updates in the time
interval $[0,t]$, both versions are known to have almost equal mixing times in \textit{total
variation} (e.g., \cite[Theorem 20.3]{levin2017markov}). Thus, for the study of total-variation
mixing time probabilists are free to switch between the discrete and continuous times as convenient.

One common approach (e.g., \cite{diaconis1996logarithmic}) to show rapid mixing of Markov processes is to prove that they ``make progress'' step-wise in terms of some $f$-divergence, most commonly the $\chi^2$ or the KL divergence.
For discrete-time Markov chains, this means that the $f$-divergence to the target distribution
decreases by a certain factor in every step, which can be described by the contraction coefficient
of the associated Markov kernel.
For continuous-time chains, this corresponds to the derivative of $f$-divergence with respect to time is suitably bounded away from zero from below, and hence the $f$-divergence to the target distribution decreases at a certain speed.
For the KL divergence, the respective contraction inequalities are known as the \textit{strong data
processing inequality (SDPI)} and the \textit{modified log-Sobolev inequality (MLSI)}, respectively for
discrete-time and continuous-time. See \cref{eqn:modified-log-sobolev-ineq,eqn:f-contraction-coef} below for formal definitions.

For a large family of Markov chains, including the Glauber dynamics on product spaces and random walks on
high-dimensional simplices (e.g., \cite{kaufman2017high}), the transition kernel $P$ consists of
two stages, for which $P=KK^*$ is written as the
product of a kernel and its adjoint. The first stage ($K$) is often described as the downward, forward, or
noising move, and the second stage ($K^*$) as the upward, backward, or denoising move.  Recent success
on analyzing such two-stage processes (e.g., \cite{anari2019log,chen2021optimal,chen2022localization}) considers the decay
of $f$-divergence in a
single stage $K$ or $K^*$, which for many specific problems turns out to
be easier to handle. For such chains it is then natural to define \textit{a half-step}
contraction coefficient (corresponding to application of $K$ only) and ask how it compares with full-step, multi-step and continuous-time
ones.

Given the equivalence between continuous-time and discrete-time mixing times, one naturally expects similar equivalence
between the contraction coefficients. This turns out to be true for the $\chi^2$ contraction (\cite[Remark 4.2]{raginsky2016strong}). Thus, it was not surprising when the work~\cite{del2003contraction} claimed to
show such equivalence for the KL divergence contraction as well (that is showing that the ratio of
the MLSI and SDPI constants is universally bounded).
{
Specifically, they claimed an equivalence (up to universal multiplicative factors) between the MLSI and the \emph{half-step SDPI}, which implies in particular that the MLSI and the \emph{full-step SDPI} (i.e., continuous-time and discrete-time entropy contraction) are equivalent. The present paper originated from us discovering a gap in their proof (see \cref{sec:comparison:error-in-dmlm}) and realizing that their claim cannot hold true because the ratio between the MLSI constant and the half-step SPDI constant can be arbitrarily large for the random transposition model (\cref{ex:random-transposition}).
However, the question of whether the MLSI and the full-step SDPI are equivalent remained open.

We answer this question by presenting an example (\cref{ex:three-state-delta-vs-rho0}) separating the MLSI and the full-step SDPI. That is, there exist cases where the contraction rate of the discrete-time chain can be much slower than that of the continuous-time one. The example is adapted from M\"{u}nch's counterexample to the Peres-Tetali conjecture \cite{Mun23}, although there seems to be no direct relationship between the properties used here and op.~cit.

The separation between continuous-time and discrete-time entropy contraction leads us to consider a related question of comparing the entropy contraction of the $m$-step kernel $P^m$ versus the $(m+1)$-step one $P^{m+1}$. If a separation is possible, then it would explain how a continuous-time chain (which averages over many $P^m$'s) could have better convergence properties at finite time than the corresponding discrete-time chain (see discussions at the end of \cref{sec:comparison:alpha-vs-delta}). It turns out that a counterexample is again possible (see \cref{ex:birth-death} below).

To avoid trivial counterexamples, we restrict our attention to \emph{factorizable} kernels, which are kernels that can be written in the form $P=KK^*$. As discussed above, this is a natural class of Markov kernels to study. All our examples are factorizable.
}

In all, the main purpose of this paper is to provide a self-contained and thorough introduction to all
common notions of relative entropy decay for finite-state Markov chains, for both continuous-time and
discrete-time versions.
We summarize known comparisons among these notions (including the recent half-step
contraction) and we give examples demonstrating that in all cases where comparisons are not
available there exist counterexamples, in the sense that the ratio can be arbitrarily large.
Along the way, we correct several misstatements that appeared in previous works, and show how to
get sharp upper and lower bounds for these coefficients and study the extremizers in the
respective functional inequalities.

\medskip
\textbf{Organization.} This paper's content is succinctly summarized in three tables: \cref{tab:const} gives
definitions of five common contraction coefficients, \cref{tab:relation} lists all known comparison inequalities
along with (new) counterexamples for the missing comparisons, and \cref{tab:example} summarizes the list of
counterexamples. The rest of the introduction defines all notions rigorously and recalls the standard
comparison chain~\cref{eqn:contraction-notion-comparison}:
$$ \rho \le \alpha \le \delta \le \rho_0 \le 2\lambda\,. $$
Then, after introducing factorizable Markov kernels ($P=KK^*$) in \cref{sec:factor},
\cref{sec:comparison} shows that every inequality above can have arbitrarily high ratio,
even when restricted to factorizable kernels. \cref{sec:extremal} concludes with some
results on the properties of extremizers of functional inequalities.

\medskip
\textbf{Notation.}
Throughout the paper, let $\cX$ be a finite set, and $P: \cX \to \cX$ be a Markov kernel with an
invariant distribution $\pi$. Assume that $(\pi, P)$ is reversible, i.e., $\pi(x)P(x,y) = \pi(y)
P(y,x)$ for all $x,y\in \cX$.
For a Markov kernel $K: \cX \to \cY$ and a distribution $\pi$ on $\cX$, we define the reverse channel $K_\pi^*: \cY\to \cX$ as
\begin{align}
  K_\pi^*(y,x) = \left\{
    \begin{array}{ll}
      \frac{\pi(x) K(x,y)}{(\pi K)(y)}, & \text{if}~(\pi K)(y)>0,\\
      \pi(x), & \text{otherwise.}
    \end{array}
  \right.
\end{align}
Note that $\pi$ is an invariant distribution of $KK_\pi^*$ and $(\pi,KK_\pi^*)$ is reversible.

\subsection{Continuous-time contraction notions} \label{sec:intro:continuous}

For $f,g: \cX \to \bR$, define the Dirichlet form as
\begin{align}
  \cE_{\pi,P}(f, g) = -\pi[f (L g)]
\end{align}
where $L=P-I$ is the Markov generator.

The \emph{Poincar\'e constant} (also called the \emph{spectral gap}) $\lambda=\lambda(\pi, P)$ is the largest number such that
\begin{align} \label{eqn:poincare-ineq}
  \lambda \Var_\pi(f) \le \cE_{\pi,P}(f,f), \qquad \forall f: \cX \to \bR,
\end{align}
where $\Var_\pi(f) := \pi[f^2] - (\pi[f])^2$ is the variance of $f$ under $\pi$.
When $f = \frac{d\nu}{d\pi}$ for some $\nu\in \cP(\cX)$ (where $\cP(\cX)$ denotes the space of distributions on $\cX$), we have
$\Var_\pi(f) = \chi^2(\nu \| \pi)$ where $\chi^2(\cdot \| \cdot)$ stands for the $\chi^2$-divergence.
\cref{eqn:poincare-ineq} is called the \emph{Poincar\'e inequality}.

The \emph{log-Sobolev constant} (LSC) $\rho=\rho(\pi, P)$ is the largest number such that
\begin{align} \label{eqn:log-sobolev-ineq}
  \rho \Ent_\pi(f) \le \cE_{\pi,P}(\sqrt f, \sqrt f), \qquad \forall f: \cX \to \bR_{\ge 0},
\end{align}
where $\Ent_\pi(f) := \pi\left[f\log \frac{f}{\pi[f]}\right]$ is the entropy of $f$.
When $f = \frac{d\nu}{d\pi}$ for some $\nu\in \cP(\cX)$, we have $\Ent_\pi(f) = D(\nu \| \pi)$ where $D(\cdot \| \cdot)$ stands for the Kullback-Leibler (KL) divergence.
\cref{eqn:log-sobolev-ineq} is called the \emph{log-Sobolev inequality} (LSI).

The \emph{modified log-Sobolev constant} (MLSC) $\rho_0=\rho_0(\pi, P)$ is the largest number such that
\begin{align} \label{eqn:modified-log-sobolev-ineq}
  \rho_0 \Ent_\pi(f) \le \cE_{\pi,P}(f, \log f), \qquad \forall f: \cX \to \bR_{\ge 0}.
\end{align}
\cref{eqn:modified-log-sobolev-ineq} is called the \emph{modified log-Sobolev inequality} (MLSI).

For reversible $(\pi, P)$, we always have (\cite{diaconis1996logarithmic,bobkov2006modified})
\begin{align} \label{eqn:rho-rho0-lambda-comparison}
  4\rho \le \rho_0 \le 2\lambda.
\end{align}
The constants $\rho$, $\rho_0$, and $\lambda$ represent the contraction ability of the continuous-time Markov chain (also known as the Markov semigroup) $(T_t)_{t\ge 0}$, where $T_t = \exp(-t L)$.\footnote{In this work we focus on continuous-time Markov chains of this type and do not consider more general continuous-time Markov chains.}
These constants are properties of the Markov generator $L$ (which can be arbitrarily scaled), rather than the Markov kernel $P$.

Let $\nu$ be a distribution on $\cX$.
Define $\nu_t = \nu T_t$ to be the distribution of the Markov chain at time $t \ge 0$ when initialized at $\nu$, and define $f_t = \frac{d\nu_t}{d\pi}$ to be the relative density.
A direct computation yields
\begin{align}
  \frac{d}{dt} \Var_\pi(f_t) &= -2\cE_{\pi,P}(f_t, f_t), \\
  \frac{d}{dt} \Ent_\pi(f_t) &= -\cE_{\pi,P}(f_t, \log f_t).
\end{align}
Therefore the Poincar\'e inequality and the modified log-Sobolev inequality can be equivalently stated as
\begin{align}
    \label{eqn:log-sobolev-ineq-alt} \frac{d}{dt} \Var_\pi(f_t) &\le -2 \lambda \Var_\pi(f_t), \\
    \label{eqn:modified-log-sobolev-ineq-alt} \frac{d}{dt} \Ent_\pi(f_t) &\le -\rho_0 \Ent_\pi(f_t).
\end{align}
\cref{eqn:log-sobolev-ineq-alt,eqn:modified-log-sobolev-ineq-alt} can also be understood as alternative definitions for $\lambda$ and $\rho_0$, from which one immediately obtains
\begin{align}
  \label{eqn:lambda-ineq} \Var_\pi(f_t) &\le \exp(-2\lambda t) \Var_\pi(f_0),  \\
  \label{eqn:rho0-ineq} \Ent_\pi(f_t) &\le \exp(-\rho_0 t) \Ent_\pi(f_0).
\end{align}

We remark that the log-Sobolev inequality is equivalent to hypercontractivity by \cite{diaconis1996logarithmic}.
Furthermore, \cite{bobkov1999exponential} shows that the log-Sobolev inequality is equivalent (up to a constant factor) to the Poincar\'e inequality on an Orlicz space. In this work we consider only the Poincar\'e inequality on the $L^2$ space.

We refer the reader to \cite{diaconis1996logarithmic,bobkov2006modified} for more discussions on (modified) log-Sobolev constants.

\subsection{Discrete-time contraction notions} \label{sec:intro:discrete}
Another class of contraction notions comes from contraction coefficients for $f$-divergences. We refer the reader to \cite[Chapter 7]{polyanskiy2024information} for an introduction to $f$-divergences.
Let $K: \cX \to \cY$ be a Markov kernel and $\pi$ be a distribution on $\cX$.
For any $f$-divergence, we define the (input-restricted\footnote{One could also consider the input-unrestricted contraction coefficient defined as $\eta_f(K) := \sup_{\pi\in \cP(\cX)} \eta_f(\pi, K)$. In this paper we focus on the input-restricted version.}) \emph{$f$-contraction coefficient} 
\begin{align} \label{eqn:f-contraction-coef}
  \eta_f(\pi, K) := \sup_{\substack{\nu\in \cP(\cX) \\ 0<D_f(\nu\|\pi)<\infty}} \frac{D_f(\nu K \| \pi K)}{D_f(\nu \| \pi)}.
\end{align}
In other words, we have
\begin{align}
  D_f(\nu K \| \pi K) \le \eta_f(\pi, K) D_f(\nu \| \pi),
\end{align}
known as the \emph{strong data processing inequality} (SDPI).
By the data processing inequality (DPI), we always have $0\le \eta_f(\pi, K) \le 1$, and a smaller value means a stronger contraction ability for $f$-divergence.
The most commonly used $f$-contraction coefficients include the total variation (TV) contraction coefficient $\eta_{\TV}(\pi, K)$, the Kullback-Leibler (KL) contraction coefficient $\eta_{\KL}(\pi, K)$ (the subscript $\KL$ is sometimes omitted), and the $\chi^2$-contraction coefficient $\eta_{\chi^2}(\pi, K)$.
It is known (\cite{ahlswede1976spreading,polyanskiy2017strong}) that
\begin{align} \label{eqn:eta-chi2-kl-tv-comparison}
  \eta_{\chi^2}(\pi, K) \le \eta_{\KL}(\pi, K) \le \eta_{\TV}(\pi, K).
\end{align}

The TV contraction coefficient $\eta_{\TV}(\pi, K)$ is also known as the Dobrushin's coefficient (\cite{dobrushin1956central}), and satisfies
\begin{align} \label{eqn:dobrushin-const-value}
  \eta_{\TV}(\pi, K) = \max_{x,x'\in \cX} \TV(K(x,\cdot),K(x',\cdot)).
\end{align}
Via \cref{eqn:eta-chi2-kl-tv-comparison}, \cref{eqn:dobrushin-const-value} provides an easy upper bound for $\eta_{\chi^2}(\pi, K)$ and $\eta_{\KL}(\pi, K)$.

We refer the reader to \cite{raginsky2016strong,polyanskiy2024information} for more discussions on contraction coefficients.


Let $P = KK_\pi^*$. We define the half-step entropy contraction coefficient
\begin{align}
  \alpha(\pi, K) = 1-\eta_{\KL}(\pi, K) = \inf_{\substack{\nu\in \cP(\cX) \\ 0<D(\nu\|\pi)<\infty}} \frac{D(\nu\|\pi)-D(\nu K\|\pi K)}{D(\nu\|\pi)}
\end{align}
and the full-step entropy contraction coefficient
\begin{align}
  \delta(\pi, P) = 1-\eta_{\KL}(\pi, P) = \inf_{\substack{\nu\in \cP(\cX) \\ 0<D(\nu\|\pi)<\infty}} \frac{D(\nu\|\pi)-D(\nu P\|\pi)}{D(\nu\|\pi)}.
\end{align}
By rearranging, we have inequalities
\begin{align}
  \label{eqn:alpha-ineq} \Ent_{\pi K}(K_\pi^*f) - \Ent_\pi(f) &\le -\alpha \Ent_\pi(f), \\
  \label{eqn:delta-ineq} \Ent_\pi(Pf) - \Ent_\pi(f) &\le -\delta \Ent_\pi(f),
\end{align}
for all $f: \cX \to \bR_{\ge 0}$.
\cref{eqn:alpha-ineq,eqn:delta-ineq} can be seen as definitions of $\alpha$ and $\delta$, and can be compared with \cref{eqn:lambda-ineq,eqn:rho0-ineq}.
In the next section, we will discuss relationships between the discrete-time contraction notions $\alpha$, $\delta$ and the continuous-time contraction notions $\rho$, $\rho_0$, $\lambda$.

\subsection{Continuous-time versus discrete-time contraction} \label{sec:intro:continuous-cs-discrete}
As we discussed above, the log-Sobolev constant $\rho(\pi, P)$, modified log-Sobolev constant $\rho_0(\pi, P)$, and the Poincar\'e constant $\lambda(\pi, P)$ represent the contraction ability of the continuous Markov chain, while the contraction coefficients represent the contraction ability of the discrete-time Markov chain.
These constants allow one to derive mixing time bounds for associated Markov chains in the continuous-time and discrete-time settings respectively, see e.g.~\cite{caputo2023lecture} and the references therein for more details.
In this paper, we give a full comparison between the discrete-time contraction notions and the continuous-time contraction notions.

The $\chi^2$-contraction coefficient has a close relationship with the Poincar\'e constant $\lambda$.
We have
\begin{align} \label{eqn:eta-chi2-K-vs-lambda}
  1-\eta_{\chi^2}(\pi, K) = \lambda(\pi, KK_\pi^*).
\end{align}
For $P=KK_\pi^*$, the two versions of contraction are equivalent up to a constant factor.
By \cref{eqn:eta-chi2-K-vs-lambda} and $1-\lambda(\pi,P^2)=(1-\lambda(\pi,P))^2$, we have
\begin{align} \label{eqn:eta-chi2-P-vs-lambda}
  1-\sqrt{\eta_{\chi^2}(\pi, P)} = \lambda(\pi, P).
\end{align}
In particular,
\begin{align} \label{eqn:eta-chi2-P-vs-lambda-no-sep}
  \frac 12 \left(1-\eta_{\chi^2}(\pi, P)\right) \le \lambda(\pi, P) \le 1-\eta_{\chi^2}(\pi, P).
\end{align}
This shows that the rates of $\chi^2$-divergence contraction for continuous-time and discrete-time Markov chains are within a factor of two of each other.

If we consider entropy (KL divergence) rather than $\chi^2$-divergence, then previous works have shown a one-sided inequality.
\cite{blanca2021entropy} shows that
\begin{align} \label{eqn:delta-vs-rho0}
  \delta(\pi, P) \le \rho_0(\pi, P).
\end{align}
\cite{miclo1997remarques} proves that
\begin{align} \label{eqn:rho-vs-alpha}
  \rho(\pi, P) \le \alpha(\pi, K).
\end{align}
By the data processing inequality, we have
\begin{align} \label{eqn:alpha-vs-delta}
  \alpha(\pi, K) \le \delta(\pi, P).
\end{align}

Summarizing the above, we have a chain of inequalities 
\begin{align} \label{eqn:contraction-notion-comparison}
  \rho \le \alpha \le \delta \le \rho_0 \le 2\lambda.
\end{align}
In other words, we have the following implications:
\begin{multline}\label{eqn:words-contraction-notion-comparison}
\text{Log-Sobolev Inequality}
\Longrightarrow
\text{Half-Step Entropy Contraction}
\Longrightarrow
\text{Full-Step Entropy Contraction} \\
\Longrightarrow
\text{Modified Log-Sobolev Inequality}
\Longrightarrow
\text{Poincar\'e Inequality}
\end{multline}
In \cref{sec:comparison}, we show that the gap between any two adjacent constants in \cref{eqn:contraction-notion-comparison} can be arbitrarily large.
In particular, unlike the $\chi^2$-divergence, the discrete-time entropy contraction $\delta(\pi, P)$ is not equivalent to the continuous-time entropy contraction $\rho_0(\pi, P)$.

\cref{tab:const} summarizes the main constants discussed in this paper, and \cref{tab:relation} summarizes relationships between these contraction notions.
While these contraction notions are in general non-equivalent (up to a constant factor), we will see in \cref{sec:comparison:other} that under extra conditions, some of them could become equivalent for certain chains.

\begin{table}[ht]
  \centering
  \begin{tabular}{|c|c|c|}
    \hline
    Symbol & Name  & Inequality (Definition) \\ \hline
    $\rho(\pi, P)$ & Log-Sobolev Constant & $\rho \Ent_\pi(f) \le \cE_{\pi,P}(\sqrt f,\sqrt f)$ \\ \hline
    $\alpha(\pi, K)$ & Half-Step Entropy Contraction & $\alpha \Ent_\pi(f) \le \Ent_\pi(f) - \Ent_{\pi K}(K_\pi^*f)$ \\ \hline
    $\delta(\pi, P)$ & Full-Step Entropy Contraction & $\delta \Ent_\pi(f) \le \Ent_\pi(f)-\Ent_\pi(Pf)$ \\ \hline
    $\rho_0(\pi, P)$ & Modified Log-Sobolev Constant & $\rho_0 \Ent_\pi(f) \le \cE_{\pi,P}(f,\log f)$ \\ \hline
    $\lambda(\pi, P)$ & Poincar\'e Constant & $\lambda \Var_\pi(f) \le \cE_{\pi,P}(f,f)$ \\ \hline
  \end{tabular}
  \caption{Contraction notions discussed in this paper. We assume that $P=KK_\pi^*$. See \cref{tab:relation} for relationships between the constants.}
  \label{tab:const}
\end{table}

\begin{table}[ht]
  \centering
  \begin{tabular}{|c|c|c|}
    \hline
    Relationship & Explanation & Reference \\ \hline
    $\rho(\pi, P) \le \alpha(\pi, K)$ & \makecell{Log-Sobolev Inequality \\ $\Rightarrow$ Half-Step Entropy Contraction} & \cite[Prop.~6]{miclo1997remarques} \\ \hline
    $\rho(\pi, P) \not \gtrsim \alpha(\pi, K)$ & \makecell{Half-Step Entropy Contraction \\ $\not \Rightarrow$ Log-Sobolev Inequality} & \cref{sec:comparison:rho-vs-alpha} \\ \hline
    $\alpha(\pi, K) \le \delta(\pi, P)$ & \makecell{Half-Step Entropy Contraction \\ $\Rightarrow$ Full-Step Entropy Contraction} & Data processing inequality \\ \hline
    $\alpha(\pi, K) \not \gtrsim \delta(\pi, P) $ & \makecell{Full-Step Entropy Contraction \\ $\not \Rightarrow$ Half-Step Entropy Contraction} & \cref{sec:comparison:alpha-vs-delta} \\ \hline
    $\delta\left(\pi, P^m\right) \not \gtrsim \delta\left(\pi, P^{m+1}\right) $ & \makecell{$(m+1)$-Step Entropy Contraction \\ $\not \Rightarrow$ $m$-Step Entropy Contraction} & \cref{sec:comparison:alpha-vs-delta} \\ \hline
    $\delta(\pi, P) \le \rho_0(\pi, P)$ & \makecell{Full-Step Entropy Contraction \\ $\Rightarrow$ Modified Log-Sobolev Inequality} & \cite[Lemma 2.7]{blanca2021entropy} \\ \hline
    $\delta(\pi, P) \not \gtrsim \rho_0(\pi, P)$ & \makecell{Modified Log-Sobolev Inequality \\ $\not \Rightarrow$ Full-Step Entropy Contraction} & \cref{sec:comparison:delta-vs-rho0} \\ \hline
    $\rho_0(\pi, P) \le 2\lambda(\pi, P)$ & \makecell{Modified Log-Sobolev Inequality \\ $\Rightarrow$ Poincar\'e Inequality} & \cite[Prop.~3.6]{bobkov2006modified} \\ \hline
    $\rho_0(\pi, P) \not \gtrsim \lambda(\pi, P)$ & \makecell{Poincar\'e Inequality \\ $\not \Rightarrow$ Modified Log-Sobolev Inequality} & \cref{sec:comparison:rho0-vs-lambda} \\ \hline
  \end{tabular}
  \caption{Relationships between contraction notions. The setting is the same as \cref{tab:const}. $a\not \gtrsim b$ means $\frac{b}{a}$ can be arbitrarily large. Inequality A $\Rightarrow$ Inequality B means Inequality A with constant $a$ implies Inequality B with constant $C a$ for some absolute constant $C>0$.}
  \label{tab:relation}
\end{table}

\section{Factorizable kernels} \label{sec:factor}
\begin{definition}[Factorizable pairs]
  Let $\cX$ be a finite set and $P: \cX \to \cX$ be a Markov kernel with invariant distribution $\pi$.
  We say $(\pi, P)$ is factorizable if $P=KK_\pi^*$ for some finite Markov kernel $K: \cX\to \cY$.
\end{definition}
\begin{lemma}[Factorizable implies reversible] \label{lem:factor-imply-reversible}
  A factorizable pair $(\pi, P)$ is reversible.
\end{lemma}
\begin{proof}
  Suppose $P=KK_\pi^*$ for a finite Markov kernel $K: \cX \to \cY$.
  For $x,y\in \cX$, we have
  \begin{align}
    \pi(x) P(x,y) &= \pi(x) \sum_{z\in \cY} K(x,z) K_\pi^*(z,y) \\
    \nonumber &= \pi(x) \pi(y) \sum_{z\in \cY} \frac 1{(\pi K)(z)} K(x,z) K(y,z)
    = \pi(y) P(y,x).
  \end{align}
  So $(\pi, P)$ is reversible.
\end{proof}

Factorizability is a reasonable assumption for several reasons. Recall that $\rho$, $\rho_0$, and $\lambda$ are properties of the Markov generator $L$, while $\alpha$ and $\delta$ are properties of the Markov kernel $P$, and the two are related by $L=P-I$. If we do not make extra assumptions on $P$, then it could happen that for two Markov kernels $P_1$ and $P_2$, $P_1$ contracts better than $P_2$, but the corresponding Markov generators $L_1$ and $L_2$ satisfy that $L_2$ contracts better than than $L_1$.
Consider the example where $\cX=[2]$, $\pi = \Unif(\cX)$, $P_c = \begin{pmatrix}
  1-c & c \\
  c & 1-c
\end{pmatrix}$ for $c\in [0, 1]$.
Then the contraction ability (as Markov kernels) is increasing for $c\in \left[0,\frac 12\right]$ and decreasing for $c\in \left[\frac 12, 1\right]$.
However, the contraction ability for the corresponding Markov generators $L_c=P_c-I$ is increasing on the whole interval $[0, 1]$.
To avoid such undesirable behavior, we need to impose extra assumptions like factorizability.
Furthermore, \cref{eqn:rho-vs-alpha,eqn:alpha-vs-delta} imply that
\begin{align} \label{eqn:rho-vs-delta}
  \rho(\pi, P) \le \delta(\pi, P)
\end{align}
for factorizable $P$. While the statement of \cref{eqn:rho-vs-delta} does not involve factorizability, it does not hold if $P$ is not factorizable.
Consider the same example with $c=1$, that is, $P=\begin{pmatrix}
  0 & 1 \\ 1 & 0
\end{pmatrix}$.
Then $\delta(\pi, P)=0$ but $\rho(\pi, P) = 1$.

\subsection{Important classes of factorizable kernels}
Another reason that the factorizability assumption is reasonable is that many natural Markov chains considered in the literature are factorizable. In this section we discuss two important classes of factorizable kernels: lazy kernels and the Glauber dynamics.
\begin{lemma}[Lazy implies factorizable] \label{lem:lazy-imply-factor}
  Let $(\pi, P)$ be a reversible pair.
  If $P(x,x) \ge \frac 12$ for all $x\in \cX$, then there exists $\cY$ and $K: \cX \to \cY$ such that $P=KK_\pi^*$.
\end{lemma}
\begin{proof}
  Let $\cY = \binom{\cX}1 \cup \binom{\cX}2$.
  Let $K: \cX \to \cY$ be the map
  \begin{align}
    K(x,e) = \left\{
      \begin{array}{ll}
        2P(x,x)-1, & \text{if}~e=\{x\},\\
        2P(x,y), & \text{if}~e=\{x,y\},\\
        0, & \text{if}~x\not \in e.
      \end{array}
    \right.
  \end{align}
  Then $K_\pi^*(e,\cdot) = \Unif(e)$.
  For $y\ne x\in \cX$, we have
  \begin{align}
    (KK_\pi^*)(x,y) = K(x,\{x,y\}) K_\pi^*(\{x,y\},y) = P(x,y).
  \end{align}
  Therefore $P=KK_\pi^*$.
\end{proof}

Lazy chains are quite common. In particular, many of our examples are lazy random walks on regular graphs.
\begin{definition}[Lazy random walk Markov chain] \label{defn:random-walk-markov-chain}
  Let $G=(V,E)$ be a $d$-regular graph.
  We associate with it a canonical Markov chain called the (lazy) random walk.
  Let $\cX=V$, $\cY=E$, and $K: \cX \to \cY$ be $K(x,e)= \frac 1d \mathbbm1\{x\in e\}$ for $x\in \cX$, $e\in \cY$.
  Then $K_\pi^*(e,\cdot) = \Unif(e)$ and $P=KK_\pi^*$ satisfies
  \begin{align}
    P(x,y) = \left\{
      \begin{array}{ll}
        \frac 12, & \text{if}~x=y,\\
        \frac 1{2d}, & \text{if}~(xy)\in E, \\
        0, & \text{otherwise.}
      \end{array}
    \right.
  \end{align}
\end{definition}

The Glauber dynamics is an important class of Markov chains that have shown huge practical and theoretical success in sampling spin systems.
Let us consider a general $\alpha$-weighted block dynamics, defined as follows.
Let $\pi$ be a probability measure on $\Omega=[q]^n$ (e.g., the Ising model on a graph with $n$ vertices). For any $\sigma,\eta\in \Omega$ and $A\subseteq [n]$, consider the conditional probability of $\eta$ given the configuration $\sigma$ on $A$:
\begin{align}
  \pi(\eta | \sigma_A) = \frac{\pi(\eta) \mathbbm1\{\sigma|_A = \eta|_A\}}{\sum_{\eta'\in \Omega} \pi(\eta') \mathbbm1\{\sigma|_A = \eta'|_A\}}.
\end{align}
Let $\cS = 2^{[n]}$ be the set of all subsets of $[n]$. For any probability measure $\alpha = (\alpha_A)_{A\in \cS}$ on $\cS$, the $\alpha$-weighted block dynamics is the Markov chain on $\Omega$ with transition matrix
\begin{align}
  P(\sigma,\sigma') = \sum_{A\in \cS} \alpha_A \pi(\sigma' | \sigma_{A^c})
\end{align}
where $A^c = [n]\backslash A$. The case $\alpha_A = \frac 1n \mathbbm1\{|A|=1\}$ is known as the Glauber dynamics. The pair $(\pi, P)$ is reversible. A factorization of the $\alpha$-weighted block dynamics can be defined as follows.
Let $\cX = \Omega$, $\cY=\cS \times \Omega$, and $K: \cX \to \cY$ be defined as
\begin{align}
  K(\sigma, (A, \eta)) = \alpha_A \pi(\eta | \sigma_{A^c}).
\end{align}
One can compute that $(\pi K)(A, \eta) = \alpha_A \pi(\eta)$ and that $K_\pi^*((A,\eta), \sigma) = \pi(\sigma | \eta_{A^c})$ for all $(A,\eta)\in \cS \times \Omega$, $\sigma\in \Omega$.
Therefore, for all $\sigma,\sigma'\in \Omega$,
\begin{align}
  \sum_{(A,\eta)\in \cS \times \Omega} K(\sigma, (A,\eta)) K_\pi^*((A,\eta), \sigma')
  = \sum_{A\in \cS} \alpha_A \sum_{\eta\in \Omega} \pi(\eta | \sigma_{A^c}) \pi(\sigma'|\eta_{A^c})
  = P(\sigma,\sigma').
\end{align}
Thus $P=KK_\pi^*$ is a factorization of the $\alpha$-weighted block dynamics.
Half-step entropy contractions for such Markov chains have been recently investigated in the context of spin systems under the name of block factorizations of entropy (e.g., \cite{BCCPSV22}).

\subsection{Non-uniqueness of factorization}
In general, a factorizable pair $(\pi, P)$ can be factorized in more than one different ways, and the associated half-step contraction rates may differ considerably. This is illustrated in the following example.
\begin{example}[Complete graph] \label{ex:complete}
  Let $n\ge 3$ be an integer. Let $\cX = [n]$, $\pi = \Unif(\cX)$.
  Let $P: \cX \to \cX$ be the lazy random walk on the complete graph.
  That is,
  \begin{align}
    P(x,y) = \left\{
      \begin{array}{ll}
        \frac 12, & \text{if}~x=y,\\
        \frac 1{2(n-1)}, & \text{if}~x\ne y.
      \end{array}
    \right.
  \end{align}
  Given an integer $2\le \ell \le n$, let $\cY$ be the set of all subsets $A\subseteq [n]$ with either $|A|=1$ or $|A|=\ell$, and define $K(\ell): \cX \to \cY$ as
  \begin{align}
    K(\ell)(x,y) = \left\{
      \begin{array}{ll}
        \frac{\ell-2}{2(\ell-1)}, & \text{if}~y=\{x\},\\
        \frac{\ell}{2(\ell-1) \binom{n-1}{\ell-1}}, & \text{if}~|y|=\ell, x\in y,\\
        0, & \text{otherwise.}
      \end{array}
    \right.
  \end{align}
  One can check that $K(\ell)_\pi^*(y, \cdot) = \Unif(y)$ and that $K(\ell) K(\ell)_\pi^* = P$.
  We also note that when $\ell=2$ this reduces to the construction in the proof of \cref{lem:lazy-imply-factor}.

  \cite[Theorem 1.1]{bristiel2024entropy} computes the half-step entropy contraction coefficient $\alpha(\pi, K(\ell))$ for every $\ell$, and shows that
  \begin{align} \label{eqn:complete-alpha}
    \alpha(\pi, K(\ell)) = \frac{\ell \log \ell}{2(\ell-1) \log n},
  \end{align}
  achieved at and only at point distributions.
  From \cref{eqn:complete-alpha} we see that $\alpha(\pi, K(\ell))$ increases with $\ell$ from the minimum value $\alpha(\pi, K(2)) = \frac{\log 2}{\log n}$ to the maximum value $\alpha(\pi, K(n)) = \frac{n}{2(n-1)}$.
  The former is of the same magnitude of the log-Sobolev constant $\rho(\pi, P) = \frac{n-2}{2(n-1)\log(n-1)}$ (\cite[Corollary A.5]{diaconis1996logarithmic}) and the latter matches asymptotically with the full-step contraction rate $\delta(\pi, P) = \frac 12 \pm o(1)$ (\cite[Eq.~(131)]{gu2023non}).
\end{example}

\subsection{Characterizations of factorizable kernels}
An interesting question is to characterize the set of factorizable kernels for a fixed $\pi$.
\cref{lem:lazy-imply-factor} provides a sufficient condition and it is certainly not necessary.
For example, any pair $(\pi, P)$ satisfying $P(x,\cdot)=\pi$ for all $x\in \cX$ is factorizable (see \cref{ex:one-step}).
On the other hand, a necessary condition for factorizability is positive semidefiniteness.
\begin{lemma}[Factorizable implies positive semidefinite] \label{lem:factor-imply-psd}
  Let $(\pi, P)$ be a reversible pair.
  If $P$ is factorizable, then the matrix $\Diag(\pi) P$ is positive semidefinite (PSD), where $\Diag(\pi)$ denotes the $\cX\times \cX$ diagonal matrix with diagonal $\pi$.
\end{lemma}
\begin{proof}
  Let $A = \Diag(\pi) P$.
  Let $K: \cX \to \cY$ be a Markov kernel such that $P = KK_\pi^*$.
  WLOG assume that $\pi K$ has full support.
  Then
  \begin{align}
    A_{x,y} = \pi(x) \sum_{z\in \cY} K(x,z) K_\pi^*(z,y)
    = \pi(x) \pi(y) \sum_{z\in \cY} \frac 1{(\pi K)(z)} K(x,z) K(y,z).
  \end{align}
  So $A = M M^\top$ where $M = \Diag(\pi) K \Diag(\pi K)^{-1/2}$.
  This finishes the proof.
\end{proof}
In particular, while a factorizable kernel is not necessarily lazy, when $\pi$ has full support, $P$ must have strictly positive diagonal entries.

For a distribution $\pi$ on $\cX$, let $\cF_\pi$ denote the set of Markov kernels $P: \cX \to \cX$ such that $(\pi, P)$ is factorizable.
\begin{lemma}[$\cF_\pi$ is convex] \label{lem:factor-is-convex}
  For any distribution $\pi$, the set $\cF_\pi$ is convex.
\end{lemma}
\begin{proof}
  Let $P_0, P_1\in \cF_\pi$ and $K_0: \cX\to \cY_0$, $K_1: \cX\to \cY_1$ be the corresponding factors.
  Let $t\in [0, 1]$. We prove that $P_t := (1-t) P_0 + t P_1$ is in $\cF_\pi$.
  Let $K_t: \cX \to \cY_0 \sqcup \cY_1$ be defined as
  \begin{align}
    K_t(x,y) = \left\{
      \begin{array}{ll}
        (1-t) K_0(x, y), & \text{if}~y\in \cY_0,\\
        t K_1(x, y), & \text{if}~y\in \cY_1.
      \end{array}
    \right.
  \end{align}
  Then we can verify that $K_t$ is a Markov kernel, and
  \begin{align}
    K_t (K_t)_\pi^* = (1-t) K_0 (K_0)_\pi^* + t K_1 (K_1)_\pi^* = P_t.
  \end{align}
\end{proof}

For a distribution $\pi$ on $\cX$, let $\cP_\pi$ denote the set of Markov kernels $P: \cX \to \cX$ such that $\Diag(\pi) P$ is PSD.
By \cref{lem:factor-imply-psd,lem:factor-is-convex}, $\cF_\pi$ is a convex subset of $\cP_\pi$.
When the state space is binary, the two sets are equal, but this is not true in general.
\begin{lemma}
  If $|\cX|=2$, then for any distribution $\pi$ on $\cX$, we have $\cF_\pi=\cP_\pi$.
\end{lemma}
\begin{proof}
  Direct calculation shows that $\cP_\pi$ is the convex hull of
  $\{J \Diag(\pi), I\}$, where $J$ is the all-ones matrix.
  Both extreme points are in $\cF_\pi$.
\end{proof}
\begin{lemma}
  For $n\ge 5$, $\cX=[n]$, and $\pi = \Unif(\cX)$, the set $\cF_\pi$ is strictly smaller than $\cP_\pi$.
\end{lemma}
\begin{proof}
  An $n\times n$ matrix $A$ is called completely positive if it can be written as $A=MM^\top$ for some (not necessarily square) matrix $M$ with non-negative entries.
  Clearly, all completely positive matrices are PSD.
  It is known (\cite{maxfield1962matrix,berman2003completely}) that for $n\le 4$, a PSD matrix is completely positive if and only if all its entries are non-negative, while for $n\ge 5$, there exist PSD matrices with strictly positive entries that are not completely positive.

  Fix $n\ge 5$, $\cX=[n]$, and $\pi=\Unif(\cX)$.
  Then $\cP_\pi$ is exactly the set of doubly stochastic PSD matrices.
  Let $A$ be an $n\times n$ PSD matrix with strictly positive entries that is not completely positive.
  By Sinkhorn's theorem (\cite{marcus1961permanent,sinkhorn1964relationship}), there is a diagonal matrix $D$ with strictly positive entries such that $DAD$ is doubly stochastic.
  Let $P=DAD$. Clearly $P$ is in $\cP_\pi$. We claim that $P$ is not in $\cF_\pi$.
  Suppose for the sake of contradiction that $P=KK_\pi^*$ for some $K$.
  Then
  \begin{align}
    A = D^{-1} P D^{-1} = \frac 1n D^{-1} K \Diag(\pi K)^{-1} K^\top D^{-1} = MM^\top
  \end{align}
  where $M=\frac 1{\sqrt n} D^{-1} K \Diag(\pi K)^{-1/2}$ is non-negative.
  This contradicts with the assumption that $A$ is not completely positive.
\end{proof}

It remains an interesting open problem to characterize $\cF_\pi$, even for uniform $\pi$.

\section{Comparison between constants} \label{sec:comparison}
In this section we compare constants in \cref{tab:const}, showing that there is a superconstant separation between any two of them. \cref{tab:example} summarizes examples in this section.

\begin{table}[ht]
  \centering
  \begin{tabular}{|c|c|c|}
    \hline
    Example & Description & Separation \\ \hline
    \cref{ex:one-step} & One-step chain & \makecell{Log-Sobolev Constant $\rho(\pi, P)$ \\ vs Half-Step Entropy Contraction $\alpha(\pi, K)$} \\ \hline
    \cref{ex:1-to-k} & $1$-to-$k$ chain & \makecell{Half-Step Entropy Contraction $\alpha(\pi, K)$ \\ vs Full-Step Entropy Contraction $\delta(\pi, P)$} \\ \hline
    \cref{ex:bernoulli-laplace} & Bernoulli-Laplace model & \makecell{Half-Step Entropy Contraction $\alpha(\pi, K)$ \\ vs Full-Step Entropy Contraction $\delta(\pi, P)$} \\ \hline
    \cref{ex:three-state-alpha-vs-delta} & Three-state chain & \makecell{One-Step Entropy Contraction $\delta(\pi, P)$ \\ vs Two-Step Entropy Contraction $\delta\left(\pi, P^2\right)$} \\ \hline
    \cref{ex:birth-death} & Birth-death chain & \makecell{$m$-Step Entropy Contraction $\delta\left(\pi, P^m\right)$ \\ vs $(m+1)$-Step Entropy Contraction $\delta\left(\pi, P^{m+1}\right)$} \\ \hline
    \cref{ex:three-state-delta-vs-rho0} & Three-state chain & \makecell{Full-Step Entropy Contraction $\delta(\pi, P)$ \\ vs Modified Log-Sobolev Constant $\rho_0(\pi, P)$} \\ \hline
    \cref{ex:expander} & Expander graph & \makecell{Modified Log-Sobolev Constant $\rho_0(\pi, P)$ \\ vs Poincar\'e Constant $\lambda(\pi, P)$} \\ \hline
    \cref{ex:random-transposition} & Random transposition model & \makecell{Half-Step Entropy Contraction $\alpha(\pi, K)$ \\ vs Modified Log-Sobolev Constant $\rho_0(\pi, P)$} \\ \hline
  \end{tabular}
  \caption{Examples in \cref{sec:comparison} and the separations they witness.}
  \label{tab:example}
\end{table}
\subsection{Log-Sobolev constant \texorpdfstring{$\rho$}{rho} vs half-step entropy contraction \texorpdfstring{$\alpha$}{alpha}} \label{sec:comparison:rho-vs-alpha}
By \cite{miclo1997remarques}, we always have $\rho(\pi, P) \le \alpha(\pi, K)$ for $P=KK_\pi^*$. The following example shows that the gap can be arbitrarily large.
\begin{example}[A one-step Markov chain] \label{ex:one-step}
  Let $\cX$ be a finite set, $\pi$ be a distribution on $\cX$ with full support.
  Let $\cY = \{*\}$ and $K: \cX \to \cY$ be the unique Markov kernel from $\cX$ to $\cY$.
  Then $P=KK_\pi^*$ satisfies $P(x,y)=\pi(y)$ for all $x,y\in \cX$.
  By \cite[Theorem A.1]{diaconis1996logarithmic},
  \begin{align}
    \rho(\pi, P) = \frac{1-2\pi_*}{\log(1/\pi_*-1)}
  \end{align}
  where $\pi_* = \min_{x\in \cX} \pi(x)$.
  On the other hand, $\alpha(\pi, K) = 1-\eta_{\KL}(\pi, K) = 1$.
  As $\pi_*\to 0$, we have $\frac{\alpha(\pi, K)}{\rho(\pi, P)} \to \infty$.
\end{example}
\subsection{Half-step entropy contraction \texorpdfstring{$\alpha$}{alpha} vs full-step entropy contraction \texorpdfstring{$\delta$}{delta}} \label{sec:comparison:alpha-vs-delta}
By the data processing inequality, we always have $\alpha(\pi, K) \le \delta(\pi, P)$ for $P=KK_\pi^*$.
\cref{ex:complete} with $\ell=2$ already shows that the gap can be arbitrarily large. Below we present a few different examples.
\begin{example}[A $1$-to-$k$ Markov chain] \label{ex:1-to-k}
  Let $\cX=[n]$, $\cY = [n]^k$, $K(x,y) = \frac 1{kn^{k-1}} \sum_{j\in [k]} \mathbbm1\{y_j=x\}$,
  $\pi = \Unif(\cX)$. In other words, on input $x\in \cX$, this chain generates a uniform length-$k$ output string and then randomly overwrite one of the $k$ positions with $x$.
  Then $K_\pi^*: \cY \to \cX$ satisfies $K_\pi^*(y,x) = \frac 1k \sum_{j\in [k]} \mathbbm1\{y_j=x\}$.
  Let $P=KK_\pi^*$.
  The motivation for this chain comes from analysis of the Glauber dynamics on $\Unif\left(\cX^k\right)$.
  The Markov kernel $K_\pi^*$ is the $k$-to-$1$ walk, and its entropy contraction is called entropic independence in \cite{AJKPV22}.
  \cref{prop:1-to-k} shows that for constant $k\ge 2$, as $n\to \infty$, we have $\frac{\delta(\pi, P)}{\alpha(\pi, K)} \to \infty$.
\end{example}

\begin{proposition}[A $1$-to-$k$ Markov chain] \label{prop:1-to-k}
  Let $\pi, K, P$ be as in \cref{ex:1-to-k}.
  Then we have $\alpha(\pi, K) = O\left(\frac{1}{\log n}\right)$ and $\delta(\pi, P) \ge 1-\frac 1k$.
  In particular, for fixed $k\ge 2$, $\frac{\delta(\pi, P)}{\alpha(\pi, K)} = \Omega(\log n)$.
\end{proposition}
\begin{proof}
  \textbf{Upper bound on $\alpha(\pi, K)$.}
  Let $\nu$ be the point distribution at $1\in \cX$.
  Then $D(\nu \| \pi) = \log n$.
  Consider the distribution $\nu K$.
  For $y\in \cY$, if $y$ contains $i$ copies of $1$, then $(\nu K)(y) = \frac{i}{k n^{k-1}}$.
  So
  \begin{align}
    D(\nu K \| \pi K) &=
    \sum_{1\le i\le k} \binom k i (n-1)^{k-i} \cdot \frac{i}{k n^{k-1}} \log \frac{n i}{k} \\
    \nonumber &= \log n - \sum_{1\le i\le k} \binom {k-1}{i-1} \frac{(n-1)^{k-i}}{n^{k-1}} \log \frac ki.
  \end{align}
  For constant $k\ge 2$, as $n\to \infty$, we have $D(\nu K \| \pi K) = \log n - \Theta(1)$.
  Therefore
  \begin{align} \label{eqn:1-to-k-alpha-upper-bound}
    \alpha(\pi, K) \le 1-\frac{D(\nu K \| \pi K)}{D(\nu \| \pi)} = \Theta\left(\frac{1}{\log n}\right).
  \end{align}

  \textbf{Lower bound on $\delta(\pi, P)$.}
  Let $M: \cY \to [k]\times [n]$ be the Markov kernel defined as $M(y,\cdot) = \Unif(\{(i,y_i) : i\in [k]\})$.
  By the data processing inequality,
  \begin{align}
    1-\delta(\pi, P) = \eta_{\KL}(\pi, P) \le \eta_{\KL}(\pi K, K_\pi^*) \le \eta_{\KL}(\pi K, M).
  \end{align}
  Let $\nu$ be any distribution on $\cY$, and $\nu_i$ ($i\in [k]$) be the $i$-th marginal of $\nu$.
  Then
  \begin{align}
    D(\nu \| \pi K) = D\left(\nu \| \pi^{\times k}\right)
    = D(\nu \| \nu_1 \times \cdots \times \nu_k) + \sum_{i\in [k]} D(\nu_i\| \pi)
    \ge \sum_{i\in [k]} D(\nu_i\| \pi).
  \end{align}
  On the other hand,
  \begin{align}
    D(\nu M \| \pi K M) = D(\nu M \| \Unif([k]) \times \pi) = \frac 1k \sum_{i\in [k]} D(\nu_i \| \pi).
  \end{align}
  Therefore
  \begin{align}
    \eta_{\KL}(\pi, M) \le \frac 1k.
  \end{align}
  So
  \begin{align} \label{eqn:1-to-k-delta-lower-bound}
    \delta(\pi, P) \ge 1-\frac 1k.
  \end{align}
\end{proof}

\begin{example}[Bernoulli-Laplace model] \label{ex:bernoulli-laplace}
  Let $n$ be a positive integer and $1\le k\le n-1$.
  We define a graph $G=(V,E)$.
  Let $V=\binom{[n]}k$ (i.e., size-$k$ subsets of of [n]).
  Equivalently, $V$ is the set of length-$n$ bit strings with Hamming weight $k$.
  There is an edge $(x,y)$ for $x,y\in V$ if and only if $\|x-y\|_1 = 2$ (considered as elements in $\{0,1\}^n$).
  The Bernoulli-Laplace model is the lazy random walk on $G$ (\cref{defn:random-walk-markov-chain}).
  That is, $\cX=V$, $\pi=\Unif(\cX)$, $\cY=E$, $K: \cX \to \cY$ is defined as $K(x,e)=\frac 1{k(n-k)} \mathbbm1\{x\in e\}$.
  For $(ij)\in \binom{[n]}{2}$, define map $\sigma_{ij}: \cX \to \cX$ by swapping the $i$-th coordinate and the $j$-th coordinate (under the bit string interpretation).
  Then
  \begin{align}
    P(x,\cdot) = \frac 12 \mathbbm1_{x} + \frac 1{2k(n-k)} \sum_{\substack{i\in x \\ j\in [n]\backslash x}} \mathbbm 1_{\sigma_{ij}(x)}.
  \end{align}
  \cref{prop:bernoulli-laplace} shows that for constant $k\ge 1$, as $n\to \infty$, we have $\frac{\delta(\pi, P)}{\alpha(\pi, K)}\to \infty$.
\end{example}

\begin{proposition}[Bernoulli-Laplace model] \label{prop:bernoulli-laplace}
  Let $\pi, K, P$ be as in \cref{ex:bernoulli-laplace} with $1\le k\le n-1$.
  Then $\alpha(\pi, K) = O\left(\frac 1{\log \binom nk}\right)$ and $\delta(\pi, P) \ge \frac{n}{2k(n-k)}$.
  In particular, for constant $k\ge 1$, we have $\frac{\delta(\pi, P)}{\alpha(\pi, K)} = \Omega(\log n)$.
\end{proposition}
\begin{proof}
  \textbf{Upper bound on $\alpha(\pi, K)$.}
  Let $\nu$ be the point distribution on any $x\in \cX$.
  Then $D(\nu \| \pi) = \log |\cX| = \log \binom nk$,
  \begin{align}
    \nu K =\frac 1{k(n-k)} \sum_{\substack{i\in x \\ j\in [n]\backslash x}} \mathbbm 1_{\{x,\sigma_{ij}(x)\}}.
  \end{align}
  Because $\pi K = \Unif(\cY)$,
  \begin{align}
    D(\nu K \| \pi K) = \log |\cY| - \log (k(n-k)) = \log \binom nk - \log 2.
  \end{align}
  Therefore
  \begin{align} \label{eqn:bernoulli-laplace-alpha-upper-bound}
    \alpha(\pi, K) \le 1-\frac{D(\nu K \| \pi K)}{D(\nu \| \pi)} = \frac{\log 2}{\log \binom nk}.
  \end{align}

  \textbf{Lower bound on $\delta(\pi, P)$.}
  We make use of the following useful result from \cite{caputo2024entropy}.
  \begin{lemma}[{\cite[Theorem 1]{caputo2024entropy}}] \label{lem:w-infty-plus-w-one-imply-delta-lower-bound}
    Let $d$ be a metric on $\cX$.
    Let $W_p$ denote the Wasserstein $p$-distance on $\cP(\cX)$
    If
    \begin{align}
      W_\infty(P(x,\cdot),P(y,\cdot)) &\le d(x,y), \qquad \forall x,y\in \cX, \label{eqn:w-infty-contraction}\\
      W_1(P(x,\cdot),P(y,\cdot)) &\le (1-\kappa)d(x,y), \qquad \forall x,y\in \cX, \label{eqn:w-1-contraction}
    \end{align}
    then
    \begin{align}
      \delta(\pi, P) \ge \kappa.
    \end{align}
  \end{lemma}

  For the Bernoulli-Laplace model, we let $d$ be the graph distance on $V=\binom{[n]}k$.
  To prove \cref{eqn:w-infty-contraction,eqn:w-1-contraction}, it suffices to prove the result for adjacent $x$ and $y$.
  By symmetry, WLOG assume that $x=\{1,3,4,\ldots,k+1\}$, $y=\{2,3,\ldots,k+1\}$.
  We define a coupling between $P(x,\cdot)$ and $P(y,\cdot)$ such that \cref{eqn:w-infty-contraction,eqn:w-1-contraction} are both satisfied.

  \begin{enumerate}[label=(\arabic*)]
    \item \label{item:bernoulli-laplace-coupling:i} For $3\le i\le k+1$, $k+2\le j\le n$, couple $\sigma_{ij}(x)$ with $\sigma_{ij}(y)$. This happens with probability $\frac{(k-1)(n-k-1)}{2k(n-k)}$ and incurs distance $1$.
    \item \label{item:bernoulli-laplace-coupling:ii} For $k+2\le j\le n$, couple $\sigma_{1j}(x)$ with $\sigma_{2j}(y)$. This happens with probability $\frac{n-k-1}{2k(n-k)}$ and incurs distance $0$.
    \item \label{item:bernoulli-laplace-coupling:iii} For $3\le i\le k+1$, couple $\sigma_{2i}(x)$ with $\sigma_{1i}(y)$. This happens with probabilities $\frac{k-1}{2k(n-k)}$ and incurs distance $0$.
    \item \label{item:bernoulli-laplace-coupling:iv} Couple $\sigma_{12}(x)$ with $y$, and $x$ with $\sigma_{12}(y)$ each with weight $\frac 1{2k(n-k)}$. This happens with probability $\frac{1}{k(n-k)}$ and incurs distance $0$.
    \item \label{item:bernoulli-laplace-coupling:v} At this point, all remaining mass in $P(x,\cdot)$ (resp.~$P(y,\cdot)$) is at $x$ (resp.~$y$). Couple them directly. This happens with probability $\frac 12 - \frac{1}{2k(n-k)}$ and incurs distance $1$.
  \end{enumerate}
  To summarize, the coupling has distance at most one and expected distance $1-\frac{n}{2k(n-k)}$.
  By \cref{lem:w-infty-plus-w-one-imply-delta-lower-bound}, we have
  \begin{align} \label{eqn:bernoulli-laplace-delta-lower-bound}
    \delta(\pi, P) \ge \frac{n}{2k(n-k)}.
  \end{align}
\end{proof}

We remark that for the Bernoulli-Laplace model, the exact value of $\alpha(\pi, K)$ has been determined in \cite[Theorem 1.12]{bristiel2024entropy}, where it is shown that \cref{eqn:bernoulli-laplace-alpha-upper-bound} is tight.
For $\delta(\pi, P)$, by considering a point distribution at any $x\in \cX$, we have
\begin{align} \label{eqn:bernoulli-laplace-delta-upper-bound}
  \delta(\pi, P) \le \frac{\log(2n(n-k))}{2 \log \binom nk}.
\end{align}
Therefore, for $n,k$ satisfying $\log k = (1-\Omega(1))\log n$, we have $\delta(\pi, P) = \Theta\left(\frac 1k\right)$.

\cref{ex:1-to-k,ex:bernoulli-laplace} show that there is a separation between $\alpha(\pi, K)$ and $\delta(\pi, P=KK_\pi^*)$. In these examples, $K_\pi^*$ can be quite different from $K$.
One natural question is whether there is a separation between $\delta(\pi, P)$ and $\delta\left(\pi, P^2\right)$. That is, can running the same Markov kernel twice result in much better contraction than running only once?
The following example shows that indeed such a separation exists. This example is adapted from \cite{Mun23}'s counterexample to the Peres-Tetali conjecture. We note, however, that the properties of this chain we use here and in \cref{ex:three-state-delta-vs-rho0} are different from those used op.~cit.

\begin{example}[A three-state Markov chain] \label{ex:three-state-alpha-vs-delta}
  Let $M$ be a positive real number.
  Let $\cX=[3]$, $\pi = \left( \frac M{M+2}, \frac 1{M+2}, \frac 1{M+2} \right)$.
  Let $P: \cX \to \cX$ be defined as
  \begin{align}
    P=\begin{pmatrix}
      1-\frac 1{4M} & \frac 1{4M} & 0 \\
      \frac 14 & \frac 12 & \frac 14 \\
      0 & \frac 14 & \frac 34
    \end{pmatrix}.
  \end{align}
  It is easy to see that $\pi$ is the invariant distribution of $P$ and $(\pi, P)$ is reversible.
  Note that by \cref{lem:lazy-imply-factor}, $P$ is factorizable.
  By \cref{prop:three-state-alpha-vs-delta}, $\frac{\delta\left(\pi, P^2\right)}{\delta(\pi, P)} \to \infty$ as $M\to \infty$.
\end{example}
\begin{proposition}[A three-state Markov chain] \label{prop:three-state-alpha-vs-delta}
  Let $\pi, P$ be as in \cref{ex:three-state-alpha-vs-delta}. Then $\delta(\pi, P) = O\left( \frac 1{\log M} \right)$ and $\delta\left(\pi, P^2\right) = \Omega(1)$.
  In particular, $\frac{\delta\left(\pi, P^2\right)}{\delta(\pi, P)} = \Omega(\log M)$.
\end{proposition}
\begin{proof}
  \textbf{Upper bound on $\delta(\pi, P)$.}
  Let $\nu$ be the point distribution at $3\in \cX$.
  Then $D(\nu \| \pi) = \log (M+2)$,
  \begin{align}
    D(\nu P \| \pi) = \frac 14 \log \frac{M+2}{4} + \frac 34 \log \frac{3(M+2)}{4} = \log (M+2) - h(1/4)
  \end{align}
  where $h(x): [0,1]\to [0,\log 2]$ is the binary entropy function
  \begin{align}
    h(x) = -x\log x - (1-x)\log (1-x).
  \end{align}
  Therefore
  \begin{align} \label{eqn:three-state-delta-P-upper-bound}
    \delta(\pi, P) \le 1-\frac{D(\nu P \| \pi)}{D(\nu \| \pi)} = O\left(\frac 1{\log M}\right).
  \end{align}

  \textbf{Lower bound on $\delta\left(\pi, P^2\right)$.}
  We prove that for large $M$, we have $P^2(x,1) = \Omega(1)$ for all $x\in \cX$.
  In fact, $P^2(1,1) \ge P(1,1)^2 = \Omega(1)$, $P^2(2,1) \ge P(2,1)P(1,1) = \Omega(1)$, and $P^2(3,1) \ge P(3,2) P(2,1) = \Omega(1)$.
  So $\TV(P^2(x,\cdot),P^2(x',\cdot)) = 1-\Omega(1)$ for all $x,x'\in \cX$.
  By \cref{eqn:dobrushin-const-value}, the Dobrushin's coefficient satisfies $\eta_{\TV}(\pi, P^2) = 1-\Omega(1)$.
  By \cref{eqn:eta-chi2-kl-tv-comparison}, we have
  \begin{align} \label{eqn:three-state-delta-PP-lower-bound}
    \delta\left(\pi, P^2\right) \ge 1-\eta_{\TV}(\pi, P^2) = \Omega(1).
  \end{align}
\end{proof}

We generalize \cref{ex:three-state-alpha-vs-delta} as follows, showing that for any positive
integer $m$, there exists a Markov kernel such that running $(m+1)$ steps results in entropy
contraction much better than running $m$ steps.
We note that there is a relatively simple characterization of the LSI for birth-death chains (\cite{chen2005poincare,chen2003variational}), but for MLSI or SDPI no such characterizations are known, except for partial progress in \cite{roberto2001inegalites,caputo2009convex}.
\begin{example}[A birth-death Markov chain] \label{ex:birth-death}
  We fix a positive integer $m$ and let $M$ be a large positive real number.
  Let $\cX = [m+2]$, $\pi(x)= \frac 1{M+m+1} + \mathbbm1\{x=1\} \frac{M-1}{M+m+1}$.
  Let $P: \cX \to \cX$ be a birth-death Markov chain, where $P(x,y)=0$ for $|x-y|\ge 2$, $P(x,x-1)=\frac 14$ for $2\le x\le m+2$, $P(x,x+1)=\frac 14$ for $2\le x\le m+1$, $P(1,2)=\frac 1{4M}$, and $P(x,x)=1-P(x,x-1)-P(x,x+1)$.
  It is easy to verify that $(\pi, P)$ is a reversible pair and $P(x,x)\ge \frac 12$ for all $x\in [m+2]$.
  By \cref{lem:lazy-imply-factor}, $(\pi, P)$ is factorizable.
  \cref{prop:birth-death} shows that as $M\to \infty$, we have $\frac{\delta\left(\pi, P^{m+1}\right)}{\delta\left(\pi, P^m\right)} \to \infty$.
\end{example}
\begin{proposition}[A birth-death Markov chain] \label{prop:birth-death}
  Let $\pi, P$ be as in \cref{ex:birth-death}.
  Then $\delta\left(\pi, P^m\right) = O\left( \frac 1{\log M}\right)$ and $\delta\left(\pi, P^{m+1}\right) = \Omega(1)$.
  In particular, $\frac{\delta\left(\pi, P^{m+1}\right)}{\delta\left(\pi, P^m\right)} = \Omega(\log M)$.
\end{proposition}
\begin{proof}
  \textbf{Upper bound on $\delta\left(\pi, P^m\right)$.}
  Let $\nu$ be the point distribution at $m+2\in \cX$. Then $D(\nu \| \pi) = \log (M+m+1)$.
  Note that $(\nu P^m)(1)=0$, $(\nu P^m)(x) = c_x$ for some $c_x = \Theta_m(1)$ for $2\le x\le m+2$, where $\Theta_m$ hides a constant factor depending on $m$. Furthermore, $(c_2,\ldots,c_{m+2})$ is a distribution on $\{2,\ldots,m+2\}$.
  Then
  \begin{align}
    D(\nu P \| \pi) = \log (M+m+1) - H(c_2,\ldots,c_{m+2})
  \end{align}
  where $H$ is the entropy function
  \begin{align}
    H(c_2,\ldots,c_{m+2}) = -\sum_{2\le i\le m+2} c_i \log c_i.
  \end{align}
  Because $c_i=\Theta_m(1)$ for all $2\le i\le m+2$, we have $H(c_2,\ldots,c_{m+2}) = \Theta_m(1)$.
  Therefore
  \begin{align} \label{eqn:birth-death-delta-m-upper-bound}
    \delta\left(\pi, P^m\right) \le 1-\frac{D(\nu P\| \pi)}{D(\nu \| \pi)} = \Theta_m\left(\frac 1{\log M}\right).
  \end{align}

  \textbf{Lower bound on $\delta\left(\pi, P^{m+1}\right)$.}
  Note that for any $x\in [m+2]$, we have
  \begin{align}
    P^{m+1}(x,1) \ge P(x,x-1)\cdots P(2,1) \cdot P(1,1)^{m+1-x} = \Omega_m(1)
  \end{align}
  where $\Omega_m$ hides a constant factor depending on $m$.
  By \cref{eqn:dobrushin-const-value,eqn:eta-chi2-kl-tv-comparison},
  \begin{align} \label{eqn:birth-death-delta-m1-lower-bound}
    \delta\left(\pi, P^{m+1}\right) \ge 1-\eta_{\TV}(\pi, P^{m+1}) = \Omega_m(1).
  \end{align}
\end{proof}

If a Markov kernel $(\pi, P)$ separates $\delta\left(\pi, P^m\right)$ and $\delta\left(\pi, P^{m+1}\right)$, then it also separates $\delta(\pi, T_t)$ (where $T_t=e^{t(P-\mathrm{Id})}$) and $\delta\left(\pi, P^m\right)$ for finite $t$.
In other words, the continuous-time chain contracts entropy at finite time better than the discrete-time counterpart.
To see this, note that for any function $f$ on $\cX$, we have
\begin{align}
  \Ent_\pi(T_t f) &\le \bE_{n\sim \Pois(t)} \Ent_\pi\left(P^n f\right) \\
  \nonumber &\le \bE_{n\sim \Pois(t)}\left[ \mathbbm{1}\{n\le m\} \Ent_\pi(f) + \mathbbm{1}\{n\ge m+1\} \Ent_\pi\left(P^{m+1} f\right)\right]\\
  \nonumber &\le \bP[\Pois(t)\le m] \Ent_\pi(f) + \bP[\Pois(t) \ge m+1] \left(1-\delta\left(\pi, P^{m+1}\right)\right) \\
  \nonumber &\le \Ent_\pi(f) \left(1 - \bP[\Pois(t)\ge m+1] \delta\left(\pi, P^{m+1}\right)\right),
\end{align}
where the first step is by convexity of $\Ent_\pi$, and the second step is by the data processing inequality.
Therefore
\begin{align}
  \delta(\pi, T_t) \ge \bP[\Pois(t)\ge m+1] \delta\left(\pi, P^{m+1}\right),
\end{align}
which is separated from $\delta\left(\pi, P^m\right)$ for finite $t$, assuming that $\delta\left(\pi, P^{m+1}\right)$ is separated from $\delta\left(\pi, P^m\right)$.

\subsection{Full-step entropy contraction \texorpdfstring{$\delta$}{delta} vs modified log-Sobolev constant \texorpdfstring{$\rho_0$}{rho0}} \label{sec:comparison:delta-vs-rho0}
By \cite[Lemma 2.7]{blanca2021entropy}, we always have $\delta(\pi, P) \le \rho_0(\pi, P)$ for any reversible
$(\pi, P)$.\footnote{They in fact prove a more general statement that works for non-reversible $(\pi, P)$.}
If we allow non-factorizable $(\pi, P)$, then it is easy to give an example where the gap is infinite: take $\cX=[2]$, $\pi=\Unif(\cX)$, and $P=\begin{pmatrix}
  0 & 1 \\ 1 & 0
\end{pmatrix}$.
The following example shows that the gap can be arbitrarily large even if we restrict to factorizable $(\pi, P)$.
\begin{example}[A three-state Markov chain] \label{ex:three-state-delta-vs-rho0}
  This example is the same as \cref{ex:three-state-alpha-vs-delta}.
  \cref{prop:three-state-delta-vs-rho0} shows that we have $\frac{\rho_0(\pi, P)}{\delta(\pi, P)} \to \infty$ as $M\to \infty$.
\end{example}
\begin{proposition}[A three-state Markov chain] \label{prop:three-state-delta-vs-rho0}
  Let $(\pi, P)$ be as in \cref{ex:three-state-alpha-vs-delta}.
  Then $\delta(\pi, P) = O\left(\frac 1{\log M}\right)$ and $\rho_0(\pi, P) = \Theta\left(\frac{\log \log M}{\log M}\right)$.
  In particular, $\frac{\rho_0(\pi, P)}{\delta(\pi, P)} = \Omega(\log \log M)$.
\end{proposition}
\begin{proof}
  \textbf{Upper bound on $\delta(\pi, P)$.}
  We have established in \cref{prop:three-state-alpha-vs-delta} that $\delta(\pi, P) = O\left(\frac 1{\log M}\right)$.

  \textbf{Lower bound on $\rho_0(\pi, P)$.}
We will show that there exists a sufficiently large universal constant $C > 0$ such that, for any positive function $f: [3] \to \bR^+$, it holds
\begin{align}\label{eq:MLSI-delta-sep}
 \Ent_\pi(f) \le \frac{C \log M}{\log\log M} \cE_{\pi,P}(f,\log f).
\end{align}
This implies $\rho_0(\pi, P) = \Omega\left( \frac{\log\log M}{\log M} \right)$.
By applying suitable scaling, we assume that $f(1) = x$, $f(2) = 1$, and $f(3) = y$ for some $x,y > 0$.
Furthermore, we may assume without loss of generality that $x \ge \frac{\beta}{M}$ where $\beta > 0$ is some tiny absolute constant, due to \cite[Lemma 2.1]{tikhomirov2024regularized}.
More specifically, \cite[Lemma 2.1]{tikhomirov2024regularized} shows that to establish a modified log-Sobolev inequality \cref{eq:MLSI-delta-sep}, it suffices to consider a restricted class of functions such that, among other restrictions, $f(1) \ge \beta' \bE_\pi f$ for an absolute constant $\beta' > 0$; such a restriction immediately implies
\begin{align}
x = f(1) \ge \beta' \left( \frac{Mx}{M+2} + \frac{1}{M+2} + \frac{y}{M+2}  \right) \ge \frac{\beta'}{M+2} \ge \frac{\beta'}{3M}.
\end{align}
Hence, we can safely assume $x \ge \frac{\beta}{M}$ where $\beta = \beta'/3$.

A direct calculation yields
\begin{align}
(M+2) \cdot \Ent_\pi(f) = Mx \log x + y \log y - (Mx + y + 1) \log\left( \frac{Mx + y + 1}{M+2} \right),
\end{align}
and
\begin{align}
4(M+2) \cdot \cE_{\pi,P}(f,\log f) = (x-1)\log x + (y-1) \log y.
\end{align}
The following claim is helpful.
\begin{claim}\label{claim:ex-three-state-delta-vs-rho0:ratio}
We have
\begin{align}
\frac{\Ent_\pi(f)}{4\, \cE_{\pi,P}(f,\log f)}
\le \frac{(2x - y - 1) + y\log y + (y+1) \log\left( \frac{1}{x} \right)}{(x-1)\log x + (y-1) \log y}.
\end{align}
\end{claim}
\begin{proof}
We rewrite the entropy as
\begin{align}
(M+2) \cdot \Ent_\pi(f) &= Mx \log x + y \log y - (Mx + y + 1) \log\left( \frac{Mx + y + 1}{M+2} \right) \\
\nonumber &= (Mx + y + 1) \log\left( \frac{(M+2)x}{Mx + y + 1} \right) + y\log y + (y+1) \log\left( \frac{1}{x} \right).
\end{align}
Notice that the first term above can be controlled by
\begin{align}
(Mx + y + 1) \log\left( \frac{(M+2)x}{Mx + y + 1} \right)
&= (Mx + y + 1) \log\left( 1 + \frac{2x - y - 1}{Mx + y + 1} \right) \\
\nonumber &\le (Mx + y + 1) \cdot \frac{2x - y - 1}{Mx + y + 1}
= 2x - y - 1.
\end{align}
The claim then follows.
\end{proof}

We consider two separate cases of $(x,y)$ to establish \cref{eq:MLSI-delta-sep}.

\textbf{Case 1: $(x,y) \notin (\frac{1}{2},\frac{3}{2}) \times (\frac{1}{2},\frac{3}{2})$.}
In this case,
we have
\begin{align}
(x-1)\log x + (y-1) \log y \ge \frac{1}{10}.
\end{align}
Since we have
\begin{align}
2x - y - 1 \le 2x-1
\le 2(x-1)\log x + 3
\le 32 \left( (x-1)\log x + (y-1) \log y \right)
\end{align}
and also
\begin{align}
y\log y \le 2(y-1) \log y + 1 \le 12 \left( (x-1)\log x + (y-1) \log y \right),
\end{align}
we deduce from \cref{claim:ex-three-state-delta-vs-rho0:ratio} that
\begin{align}
\frac{\Ent_\pi(f)}{4\, \cE_{\pi,P}(f,\log f)}
\le 44 + \frac{(y+1) \log\left( \frac{1}{x} \right)}{(x-1)\log x + (y-1) \log y}.
\end{align}
Consider three subcases.

\begin{enumerate}[label=(\roman*)]
  \item If $x \ge \frac{3}{2}$, then $\log(1/x) < 0$ and hence
  \begin{align}
  \frac{\Ent_\pi(f)}{4\, \cE_{\pi,P}(f,\log f)}
  \le 44.
  \end{align}

  \item If $x \le \frac{1}{2}$, then consider how large $y$ is.
  If $y \le \frac{\log M}{\log\log M}$, then by $(x-1)\log x \ge \frac{1}{2} \log(1/x)$ and $(y-1)\log y \ge 0$ we deduce that
  \begin{align}
  \frac{(y+1) \log\left( \frac{1}{x} \right)}{(x-1)\log x + (y-1) \log y} \le 2(y+1)
  \le \frac{4\log M}{\log\log M}.
  \end{align}
  If $y > \frac{\log M}{\log\log M}$, then by $(x-1)\log x \ge 0$, $\frac{y+1}{y-1} \le 3$, and $x \ge \frac{\beta}{M}$ we deduce that
  \begin{align}
  \frac{(y+1) \log\left( \frac{1}{x} \right)}{(x-1)\log x + (y-1) \log y}
  \le \frac{(y+1) \log\left( \frac{1}{x} \right)}{(y-1) \log y}
  \le \frac{3 \log(\frac{M}{\beta})}{\log (\frac{\log M}{\log\log M})}
  \le \frac{C_0\log M}{\log\log M},
  \end{align}
  for some $C_0 = C_0(\beta) > 0$ when $M$ is sufficiently large.

  \item If $x \in (\frac{1}{2},\frac{3}{2})$, then by our assumption it must hold $y \notin (\frac{1}{2},\frac{3}{2})$.
  Since $\log(1/x) \le 1$ when $x > \frac{1}{2}$, we have
  \begin{align}
  \frac{(y+1) \log\left( \frac{1}{x} \right)}{(x-1)\log x + (y-1) \log y}
  \le \frac{y+1}{(y-1) \log y}
  \le 10,
  \end{align}
  when $y \notin (\frac{1}{2},\frac{3}{2})$.
\end{enumerate}

Therefore, in all three subcases we have
\begin{align}
\frac{\Ent_\pi(f)}{4\, \cE_{\pi,P}(f,\log f)}
\le \frac{C\log M}{\log\log M}
\end{align}
where $C = C(\beta) > 0$ is constant, whenever $M$ is sufficiently large.

\textbf{Case 2: $(x,y) \in (\frac{1}{2},\frac{3}{2}) \times (\frac{1}{2},\frac{3}{2})$.}
In this case,
we have
\begin{align}
(x-1)\log x + (y-1) \log y \ge \frac{1}{2} (x-1)^2 + \frac{1}{2} (y-1)^2,
\end{align}
and also
\begin{align}
&~(2x - y - 1) + y\log y + (y+1) \log\left( \frac{1}{x} \right) \\
\nonumber \le&~2(x-1) - (y-1) + y(y-1) + (y+1) \left( \frac{1}{x}-1 \right) \\
\nonumber =&~\frac{1}{x} (x-1)\left( 2x - y-1 \right) + (y-1)^2 \\
\nonumber =&~\frac{2}{x}(x-1)^2 - \frac{1}{x}(x-1)(y-1) + (y-1)^2 \\
\nonumber \le&~4(x-1)^2 + 2|(x-1)(y-1)| + (y-1)^2 \\
\nonumber \le&~5(x-1)^2 + 5(y-1)^2.
\end{align}
By \cref{claim:ex-three-state-delta-vs-rho0:ratio},
\begin{align}
\frac{\Ent_\pi(f)}{4\, \cE_{\pi,P}(f,\log f)} \le 10.
\end{align}

Combining the two cases, we conclude that
\begin{align} \label{eqn:three-state-rho0-lower-bound}
  \rho_0(\pi, P) = \Omega\left( \frac{\log \log M} {\log M}\right).
\end{align}

In fact, this lower bound on $\rho_0(\pi, P)$ is asymptotically tight and can be achieved by, for example, $f(1)=1/M$, $f(2)=1$, and $f(3)=\log M$ as given in \cite{Mun23}.
Therefore,
\begin{align}
  \rho_0(\pi, P) = \Theta\left( \frac{\log\log M}{\log M} \right)
\end{align}
as $M\to \infty$.

\end{proof}

\subsection{Modified log-Sobolev constant \texorpdfstring{$\rho_0$}{rho0} vs Poincar\'e constant \texorpdfstring{$\lambda$}{lambda}} \label{sec:comparison:rho0-vs-lambda}
\cite{bobkov2006modified} shows that $\rho_0(\pi, P) \le 2\lambda(\pi, P)$, and gave an example showing that the gap can be arbitrarily large. We record their example here.
\begin{example}[Expander graphs] \label{ex:expander}
  Let $G=(V,E)$ be a expander graph with bounded degree.
  Consider the lazy random walk on $G$ (\cref{defn:random-walk-markov-chain}).
  \cite{bobkov2006modified} shows that $\rho_0(\pi, P) = \Theta\left(\frac 1{\log |V|}\right) = \Theta\left(\frac{\lambda_1(\pi, P)}{\log |V|}\right)$.
  Therefore as $|V|\to \infty$, we have $\frac{\lambda_1(\pi, P)}{\rho_0(\pi, P)} \to \infty$.
\end{example}

\subsection{Other comparisons} \label{sec:comparison:other}
\cite{diaconis1996logarithmic} shows that the spectral gap $\lambda$ and the log-Sobolev constant $\rho$ differ by at most a factor of $O(\log(1/\pi_{\min}))$ where
\begin{align}
    \pi_{\min} = \min_{x \in \cX:\, \pi(x) > 0} \pi(x).
\end{align}
More precisely, \cite[Corollay A.4]{diaconis1996logarithmic} shows that, assuming $\pi_{\min} \le 1/2$, it holds
\begin{align}
    \frac{\lambda}{2 + \log(1/\pi_{\min})} \le
    \frac{(1-2\pi_{\min})\lambda}{\log(1/\pi_{\min} - 1)} \le \rho \le \frac{\lambda}{2}.
\end{align}
This in particular shows that all constants discussed in this paper, including also the half-step entropy contraction $\alpha$, the full-step entropy contraction $\delta$, and the modified log-Sobolev constant $\rho_0$, differ by at most a factor of $O(\log(1/\pi_{\min}))$ from each other.

In a recent work \cite{salez2023upgrading}, Salez, Tikhomirov, and Youssef establish a surprising and remarkable comparison between the modified log-Sobolev constant $\rho_0$ and the log-Sobolev constant $\rho$. For a reversible Markov kernel $P$ with respect to a probability measure $\pi$, define the sparsity parameter as
\begin{align}
    p_{\min} = \min_{(x,y) \in \cX^2:\, P(x,y) > 0} P(x,y).
\end{align}
Then, \cite[Theorem 1]{salez2023upgrading} shows that
\begin{align}
    \frac{\rho_0}{20 \log(1/p_{\min})} \le \rho \le \frac{\rho_0}{4}.
\end{align}
Hence, all entropy-related constants discussed in this paper, including also the half-step entropy contraction $\alpha$ and the full-step entropy contraction $\delta$, differ by at most a factor of $O(\log(1/p_{\min}))$ from each other.

\subsection{Comments on several previous works} \label{sec:comparison:error-in-dmlm}
\cite[Prop.~5.1]{del2003contraction} claims that
\begin{align} \label{eqn:alpha-vs-rho0-dmlm03}
  c \rho_0 \le \alpha
\end{align}
for some universal constant $c>0$.
{
In fact, \cref{eqn:alpha-vs-rho0-dmlm03} fails for the random transposition model, showing that the claim must be incorrect.
\begin{example}[Random transposition model] \label{ex:random-transposition}
  Let $n$ be a positive integer.
  Let $G=(V,E)$ be the Cayley graph on the symmetric group $S_n$ generated by transpositions. That is, $V$ is the set of permutations of $[n]$, and there is an edge between $\sigma,\tau\in V$ if and only if they differ by exactly two entries. The random transposition model is the lazy random walk on $G$ (\cref{defn:random-walk-markov-chain}).
  Considering the point measure at any $x\in \cX$ gives $\alpha\le \frac{\log 2}{\log (n!)} = \Theta\left(\frac 1{n\log n}\right)$.
  In fact, \cite[Theorem 1.9]{bristiel2024entropy} shows that this is tight, i.e, $\alpha = \frac{\log 2}{\log (n!)}$.
  On the other hand, \cite{gao2003exponential} shows that $\rho_0 = \Theta\left(\frac 1n\right)$.
  Therefore, as $n\to \infty$, we have $\frac{\rho_0}{\alpha} \to \infty$.
\end{example}
Our separations of $\alpha$ vs $\delta$ (\cref{sec:comparison:alpha-vs-delta}) and $\delta$ vs $\rho_0$ (\cref{sec:comparison:delta-vs-rho0}) also provide alternative counterexamples to the claim.
We now explain briefly the issue in the proof of \cref{eqn:alpha-vs-rho0-dmlm03} in \cite{del2003contraction}.
}
In their proof of Prop.~5.1, the authors apply a technical result, Lemma 5.2, which represents the (relative) entropy of a function $f$ with expectation $\bE_n(f)=1$ (where $n$ denotes the underlying probability measure) as an integral of the covariance between $f_t$ and $\log(f_t)$ where $f_t = e^{-t} f + (1-e^{-t})$, $t \in \bR_{\ge 0}$ represents an interpolation between $f$ and $1$.
However, in the actual application of Lemma 5.2, the measure $n$ is a conditional probability measure under which the expectation of $f$ is no longer $1$, and hence the interpolation function $f_t$ should be replaced by $f_t = e^{-t} f + (1-e^{-t}) \bE_n(f)$; this would require the function $f_t$ to depend on the conditioning and the proofs following afterwards no longer work.

\cite[Prop.~4.3]{raginsky2016strong} claims that
\begin{align} \label{eqn:alpha-vs-rho0-rag16}
  \rho_0 \le 1-c(1-\alpha)
\end{align}
for some universal constant $c>0$. Our examples do not disprove the claim. However, the proof of \cite[Theorem 4.4]{raginsky2016strong} is a generalization of that of \cite[Prop.~5.1]{del2003contraction}, so it has the same error. In particular, the last display of the proof of \cite[Prop.~4.3]{raginsky2016strong} implies that $c \rho_0 \le \alpha$ for some constant $c>0$, which we have shown to be incorrect. Therefore the proof of \cite[Prop.~4.3]{raginsky2016strong} is incorrect and does not establish \cref{eqn:alpha-vs-rho0-rag16}. It is unclear whether \cref{eqn:alpha-vs-rho0-rag16} as stated is correct.


\section{Extremal functions} \label{sec:extremal}
In this section we discuss another difference between the continuous-time entropy contraction constants and the discrete-time entropy contraction constants.
It is known (\cite{bobkov2006modified}) that for irreducible $(\pi, P)$, the log-Sobolev constant $\rho(\pi, P)$ and the modified log-Sobolev constant $\rho_0(\pi, P)$ satisfy a dichotomy: they are either equal to twice the Poincar\'e constant $\lambda(\pi, P)$ or achieved at a full-support function.
We show that this is no longer true for the discrete-time entropy contraction constants $\alpha(\pi, K)$ and $\delta(\pi, P)$ by providing explicit examples whose extremal functions have non-full support.


\subsection{Log-Sobolev constant \texorpdfstring{$\rho$}{rho}, modified log-Sobolev constant \texorpdfstring{$\rho_0$}{rho0}, Poincar\'e constant \texorpdfstring{$\lambda$}{lambda}} \label{sec:extremal:rho-rho0-lambda}
\cite{bobkov2006modified} studies extremal functions for $\rho$, $\rho_0$, $\lambda$.
The extremal functions for the Poincar\'e constant $\lambda(\pi, P)$ are easy to describe.
They are the (right) eigenfunctions of $-L$ corresponding to the eigenvalue $\lambda$.
In particular, $\lambda$ is always achieved at some non-constant function $f: \cX \to \bR$.

For the log-Sobolev constant $\rho$, \cite{bobkov2006modified} shows that for any reversible $(\pi, P)$, either
\begin{enumerate}[label=(\roman*)]
  \item $\rho(\pi, P) = 2 \lambda(\pi, P)$, or
  \item there exists a non-constant function $f: \cX \to \bR_{\ge 0}$ with $\pi[f]=1$ such that
  \begin{align}
    \rho \Ent_\pi(f) = \cE_{\pi,P}(\sqrt f, \sqrt f).
  \end{align}
  Furthermore, such $f$ satisfies the equation
  \begin{align}
    -L\sqrt f = \rho \sqrt f \log f.
  \end{align}
  If $(\pi, P)$ is irreducible, then $f$ has full support.
\end{enumerate}

For the modified log-Sobolev constant $\rho_0(\pi, P)$, \cite{bobkov2006modified} shows that for any reversible $(\pi, P)$, either
\begin{enumerate}[label=(\roman*)]
  \item $\rho_0(\pi, P) = 2 \lambda(\pi, P)$, or
  \item there exists a non-constant function $f: \cX \to \bR_{\ge 0}$ with $\pi[f]=1$ such that
  \begin{align}
    \rho_0 \Ent_\pi(f) = \cE_{\pi,P}(f, \log f).
  \end{align}
  Furthermore, such $f$ satisfies the equation
  \begin{align}
    -Lf -f L(\log f)=\rho_0 f \log f.
  \end{align}
  If $(\pi, P)$ is irreducible, then $f$ has full support.
\end{enumerate}

\subsection{Half-step entropy contraction \texorpdfstring{$\alpha$}{alpha} and full-step entropy contraction \texorpdfstring{$\delta$}{delta}} \label{sec:extremal:alpha-delta}
In this section we study the extremal distributions for $\eta_{\KL}$, which includes both $\alpha$ and $\delta$.
\begin{lemma}[Extremal distributions for $\eta_{\KL}$] \label{lem:extremal-eta-KL}
  Let $\pi$ be a distribution on $\cX$ and $K: \cX \to \cY$ be a Markov kernel. Then either
  \begin{enumerate}[label=(\roman*)]
    \item \label{item:lem-extremal-eta-KL:i} $\eta_{\KL}(\pi, K) = \eta_{\chi^2}(\pi, K)$, or
    \item \label{item:lem-extremal-eta-KL:ii} there exists distribution $\nu$ on $\cX$ such that
    \begin{align}
      \eta_{\KL}(\pi, K) = \frac{D(\nu K \| \pi K)}{D(\nu \| \pi)}.
    \end{align}
  \end{enumerate}
\end{lemma}
\begin{proof}
  WLOG assume that $\pi$ has full support.
  If \cref{item:lem-extremal-eta-KL:ii} does not happen, then there exists a sequence $\{f_n\}_n$ ($f_n: \cX\to \bR_{\ge 0}$, $\pi[f_n]=1$) satisfying $\|f_n-\mathbbm1\|_\infty \to 0$ and $\frac{\Ent_{\pi K}(K_\pi^* f_n)}{\Ent_\pi(f_n)} \to \eta_{\KL}(\pi, K)$ as $n\to \infty$.

  Write $f_n = \mathbbm1+\epsilon_n g_n$, where $\epsilon_n\ge 0$, $\pi[g_n^2]=1$.
  Note that the space $\cG:=\{g: \cX \to \bR: \pi[g]=0, \pi[g^2]=1\}$ is compact.
  By replacing the sequence $\{f_n\}_n$ with a subsequence, we can WLOG assume that there exists $g^*\in \cG$ such that $\|g_n-g^*\|_\infty \to 0$ as $n\to \infty$.

  Now let us prove that
  \begin{align}
    \lim_{n\to \infty} \frac{\Ent_{\pi K}(K_\pi^* f_n)}{\Ent_\pi(f_n)} = \frac{\Var_{\pi K}(K_\pi^* g^*)}{\Var_\pi(g^*)}. \label{eqn:lem-extremal-eta-KL:lim-ratio}
  \end{align}
  If \cref{eqn:lem-extremal-eta-KL:lim-ratio} holds, then
  \begin{align}
    \eta_{\chi^2}(\pi, K) = \sup_{g\in \cG} \frac{\Var_{\pi K}(K_\pi^* g)}{\Var_\pi(g)} \ge \eta_{\KL}(\pi, K),
  \end{align}
  and \cref{item:lem-extremal-eta-KL:i} holds.

  By \cref{lem:lem-extremal-eta-KL:helper},
  \begin{align}
    &~ \left\|(1+\epsilon_n g_n) \log (1+\epsilon_n g_n) - (1+\epsilon_n g^*)\log(1+\epsilon_n g^*)  - \epsilon_n (g_n-g^*)\right\|_\infty\\
    \nonumber =&~ O\left(\epsilon_n^2 (\epsilon_n+\|g_n-g^*\|_\infty)\right).
  \end{align}

  Taking expectation and using triangle inequality, we get
  \begin{align}
    \left|\Ent_\pi(f_n) - \Ent_\pi(1+\epsilon_n g^*)\right| = O\left(\epsilon_n^2 (\epsilon_n+\|g_n-g^*\|_\infty)\right).
  \end{align}

  It is known that
  \begin{align}
    \lim_{\epsilon\to 0} \frac 1{\epsilon^2} \Ent_\pi(1+\epsilon g^*) = \frac 12 \Var_\pi(g^*).
  \end{align}
  So
  \begin{align}
    \lim_{n\to \infty} \frac 1{\epsilon_n^2} \Ent_\pi(f_n) = \frac 12\Var_\pi(g^*).
  \end{align}
  Similarly
  \begin{align}
    \lim_{n\to \infty} \frac 1{\epsilon_n^2} \Ent_{\pi K}(K_\pi^* f_n) = \frac 12\Var_{\pi K}(K_\pi^* g^*).
  \end{align}
  This finishes the proof of \cref{eqn:lem-extremal-eta-KL:lim-ratio}.
\end{proof}

\begin{lemma} \label{lem:lem-extremal-eta-KL:helper}
  There exists $C>0$ such that for $\epsilon, \epsilon'>0$ small enough, for $x,y,z,w\in \bR$ satisfying $|x|,|y|,|z|,|w|\le \epsilon$, $|x-z|,|y-w|\le \epsilon\epsilon'$, we have
  \begin{align}
      |(1+x)\log(1+y)-(1+z)\log(1+w) - (y-w)| \le C \epsilon^2 (\epsilon+\epsilon').
  \end{align}
\end{lemma}
\begin{proof}
  First note that $|\log(1+y)-(y-y^2/2)| = O(\epsilon^3)$.
  Then
  \begin{align}
  &~ |(1+x)\log(1+y)-(1+z)\log(1+w) - (y-w)| \\
  \nonumber =&~ |(1+x)(y-y^2/2) - (1+z) (w-w^2/2) - (y-w)| + O(\epsilon^3) \\
  \nonumber =&~ |xy-zw| + |y^2-w^2|/2 + |xy^2|/2 + |zw^2|/2+ O(\epsilon^3) \\
  \nonumber =&~ |(x-z)y + z(y-w)| + |(y-w)(y+w)|/2 + O(\epsilon^2(\epsilon+\epsilon')) \\
  \nonumber =&~ O(\epsilon^2(\epsilon+\epsilon')).
  \end{align}
\end{proof}

Unlike $\rho$ and $\rho_0$ where the extremal functions (if they exist) have full support, the extremal distributions for $\eta_{\KL}$ may have non-full support. 

\begin{example}[Complete graph] \label{ex:coloring}
  Let $n\ge 3$ be an integer.
  Let $\cX = [n]$ and $\pi = \Unif(\cX)$.
  Let $K: \cX \to \cX$ be the (non-lazy) random walk on the complete graph. That is,
  \begin{align}
    K(x,y) = \frac 1{n-1} \mathbbm1\{x\ne y\}
  \end{align}
  for $x,y\in \cX$.
  \cite[Prop.~33]{gu2023non} proves that $\eta_{\KL}(\pi, K) = \frac{\log n-\log(n-1)}{\log n}$, and is achieved at and only at point distributions.
\end{example}
In the above example, $K$ is not factorizable, so it corresponds to the half-step contraction coefficient $\alpha$.
The next example shows the extremal distributions for the full-step contraction coefficient $\delta$ can have non-full support.
\begin{example}[Complete bipartite graph] \label{ex:complete-bipartite}
  Let $n\ge 3$ be an integer.
  Let $G=K_{n,n}$ be the complete bipartite graph.
  That is, $V=[2]\times [n]$, and there is an edge between $(1,i)$ and $(2,j)$ for all $i,j\in [n]$.
  Let $(\pi, P)$ be the random walk on $G$ (\cref{defn:random-walk-markov-chain}).
  Numerical computation suggests that $\eta_{\KL}(\pi, P)=\frac{\log n}{2\log(2n)}$, and equality is achieved at and only at point distributions. \cref{prop:complete-bipartite} proves this observation for $n=3$.
\end{example}
\begin{proposition}[Complete bipartite graph] \label{prop:complete-bipartite}
  Let $\pi, P$ be as in \cref{ex:complete-bipartite}. For $n=3$, $\eta_{\KL}(\pi, P)=\frac{\log n}{2\log(2n)}$, and equality is achieved at and only at point distributions.
\end{proposition}
\begin{proof}
  Let $\nu$ be the point distribution on any $x\in \cX$.
  Then
  \begin{align}
    D(\nu \| \pi) &= \log(2n),\\
    D(\nu P \| \pi) &= \frac 12 \log n.
  \end{align}
  So
  \begin{align}
    \eta_{\KL}(\pi, P) \ge \frac{D(\nu P \| \pi)}{D(\nu \| \pi)} = \frac{\log n}{2\log(2n)}.
  \end{align}
  Note that this holds for any $n\ge 3$.

  We prove that for $n=3$, the point distributions are the only maximizers.
  We represent a distribution $\nu$ using a tuple $(t,\nu_1,\nu_2)$, where $t\in [0,1]$ and $\nu_i$ ($i=1,2$) is a distribution on $\cX_i=\{i\}\times [n] \subseteq \cX$.
  Given $\nu$, the corresponding tuple $\phi(\nu) = (\nu(\cX_1), \nu|_{\cX_1}, \nu|_{\cX_2})$, where $\nu|_{\cX_i}$ denotes the conditional distribution (if $\nu(\cX_i)=0$, then choose an arbitrary distribution on $\cX_i$).
  Given a tuple $(t,\nu_1,\nu_2)$, the corresponding distribution $\nu$ is $\psi(t,\nu_1,\nu_2) = t \nu_1 + (1-t) \nu_2$.
  Let $\pi_i=\Unif(\cX_i)$ for $i=1,2$. By symmetry, we only need to consider the case $\frac 12 \le t \le 1$.

  Under the tuple parametrization, we have
  \begin{align}
    D(\psi(t,\nu_1,\nu_2) \| \pi) &= d_{\KL}\left(t\|\frac 12\right) + t D(\nu_1 \| \pi_1) + (1-t) D(\nu_2 \| \pi_2),\\
    \phi(\psi(t,\nu_1,\nu_2) P) &= \left(\frac 12, t \nu_1 + (1-t) \pi_1, (1-t) \nu_2 + t \pi_2\right),
  \end{align}
  where $d_{\KL}(x\|y) = x \log \frac x y + (1-x) \log \frac {1-x}{1-y}$ is the binary KL divergence function.
  So for $\nu = \psi(t,\nu_1,\nu_2)$, we have
  \begin{align} \label{eqn:complete-bipartite-ent-ratio}
    \frac{D(\nu P\| \pi)}{D(\nu \| \pi)} = \frac 12 \cdot \frac{D(t \nu_1 + (1-t) \pi_1 \| \pi_1) + D((1-t) \nu_2 + t \pi_2 \| \pi_2)}{d_{\KL}\left(t\|\frac 12\right) + t D(\nu_1 \| \pi_1) + (1-t) D(\nu_2 \| \pi_2)}.
  \end{align}

  When $t=1$, \cref{eqn:complete-bipartite-ent-ratio} simplifies to
  \begin{align}
    \frac{D(\nu P\| \pi)}{D(\nu \| \pi)} = \frac 12 \cdot \frac{D(\nu_1 \| \pi_1)}{\log 2 + D(\nu_1 \| \pi_1)}
  \end{align}
  which is maximized when and only when $\nu_1$ is a point measure, taking value $\frac{\log n}{2\log(2n)}$.
  Note that when $t=1$ and $\nu_1$ is a point measure, $\nu=\psi(t,\nu_1,\nu_2)$ does not depend on $\nu_2$ and is always a point measure.

  Now assume that $\frac 12 \le t<1$. By \cite[Prop.~17]{gu2023non}, for fixed $D(\nu_1 \| \pi_1)$ (resp.~$D(\nu_2 \| \pi_2)$), $D(t \nu_1 + (1-t)\pi_1 \| \pi_1)$ (resp.~$D(\nu_2 + t \pi_2 \| \pi_2)$) is uniquely (up to permutation of the alphabet) at a distribution of form $\left(x, \frac{1-x}{n-1}, \ldots, \frac{1-x}{n-1}\right)$ where $\frac 1n \le x \le 1$.
  Therefore, we can assume WLOG that $\nu_1 = \left(x, \frac{1-x}{n-1}, \ldots, \frac{1-x}{n-1}\right)$, $\nu_2 = \left(y, \frac{1-y}{n-1}, \ldots, \frac{1-y}{n-1}\right)$ for some $x,y\in \left[\frac 1n, 1\right]$.
  In this case, we can simplify \cref{eqn:complete-bipartite-ent-ratio} as
  \begin{align}
    \frac{D(\nu P\| \pi)}{D(\nu \| \pi)} = \frac 12 \cdot \frac{f_n\left(t x+\frac{1-t}{n}\right) + f_n\left((1-t)y+\frac t n\right)}{f_2(t) + t f_n(x) + (1-t) f_n (y)} =: F_n(t,x,y),
  \end{align}
  where
  \begin{align}
    f_n(x) = D\left( \left(x, \frac{1-x}{n-1}, \ldots, \frac{1-x}{n-1}\right) \| \Unif([n]) \right)
    = \log n + x \log x + (1-x) \log \frac{1-x}{n-1}.
  \end{align}

  Therefore, computing $\eta_{\KL}(\pi, P)$ is equivalent to computing
  \begin{align} \label{eqn:complete-bipartite-three-var}
    \sup_{\substack{\frac 12 \le t\le 1\\ \frac 1n \le x,y\le 1\\ (t,x,y) \ne \left(\frac 12, \frac 1n, \frac 1n\right)}} F_n(t,x,y)
  \end{align}

  We have reduced the original problem of computing $\eta_{\KL}(\pi, K)$, which is a priori an $(2n-1)$-dimensional optimization problem, to a optimization problem with three real variables $t \in [0, 1], x,y \in \left[\frac 1n, 1\right]$. The following lemma helps us reduce it further to two real variables.

  \begin{lemma} \label{lem:complete-bipartite-1}
    For any $n\ge 3$, there exists $t^*>\frac 12$ such that
    \begin{align}
      \sup_{\substack{0< t < t^* \\ \frac 1n < x \le 1}} \frac {f_n\left( t x + \frac{1-t}n \right)}{t f_n(x)} < \frac{\log n}{\log (2n)}.
    \end{align}
  \end{lemma}
  \begin{proof}
    By \cite[Eqs.~(23) and (27)]{gu2023non}, for fixed $0\le t\le 1$, we have
    \begin{align}
      \sup_{\frac 1n < x \le 1} \frac {f_n\left( t x + \frac{1-t}n \right)}{t f_n(x)}
      \le t^{\frac 1{\log n}}.
    \end{align}
    For any $t < \exp\left((\log n) \log \frac{\log n}{\log(2n)} \right)$,
    we have
    \begin{align}
      t^{\frac 1{\log n}} < \frac{\log n}{\log(2n)}.
    \end{align}
    Furthermore, notice that $\exp\left((\log n) \log \frac{\log n}{\log(2n)} \right) > \frac 12$ for all $n\ge 3$.
    So we can take $t^*$ to be any number smaller than $\exp\left((\log n) \log \frac{\log n}{\log(2n)} \right)$.
  \end{proof}

  Note that all arguments until this point work for any $n\ge 3$. From now on we will use the assumption $n=3$.
  For $n=3$, we take $t^* = 0.58 < \exp\left((\log 3) \log \frac{\log 3}{\log 6} \right)$ and apply \cref{lem:complete-bipartite-1} to \cref{eqn:complete-bipartite-three-var}.
  For $\frac 12 \le t \le t^*$, we can apply \cref{lem:complete-bipartite-1} to $(t,x)$ and $(1-t,y)$ and get
  \begin{align} \label{eqn:complete-bipartite-three-var-i}
    \sup_{\substack{\frac 12 \le t\le t^*\\ \frac 13 \le x,y\le 1\\ (t,x,y) \ne \left(\frac 12, \frac 13, \frac 13\right)}} F_3(t,x,y) < \frac{\log 3}{2\log 6}
  \end{align}
  For $t^* < t < 1$, we have
  \begin{align}
    \sup_{\substack{t^* < t < 1\\ \frac 13 \le x,y\le 1}} F_3(t,x,y) \ge \lim_{t\to 1^-} F_3\left(t,1,\frac 13\right) = \frac{\log 3}{2\log 6}.
  \end{align}
  Applying \cref{lem:complete-bipartite-1} to $(1-t,y)$ we see that for any $(t,x,y)$ with $t^*<t<1$ and $F_3(t,x,y) \ge \frac{\log 3}{2\log 6}$, we have $F_3(t,x,y) \le F_3\left(t,x,\frac 13\right)$.
  So
  \begin{align} \label{eqn:complete-bipartite-three-var-ii}
    \sup_{\substack{t^* < t < 1\\ \frac 13 \le x,y\le 1}} F_3(t,x,y) =
    \sup_{\substack{t^* < t < 1\\ \frac 13 \le x\le 1}} F_3\left(t,x,\frac 13\right)
    =
    \sup_{\substack{t^* < t < 1\\ \frac 13 \le x\le 1}} \frac 12 \cdot \frac{f_3\left(t x+\frac{1-t}{3}\right)}{f_2(t) + t f_3(x)}.
  \end{align}

  In the following, we prove that for $t^* < t < 1$, $\frac 13 \le x \le 1$, we have
  \begin{align} \label{eqn:complete-bipartite-two-var}
    G_3(t,x) = \frac{\log 6}{\log 3} \cdot f_3\left(t x+\frac{1-t}{3}\right) - (f_2(t) + t f_3(x)) < 0.
  \end{align}
  Note that \cref{eqn:complete-bipartite-two-var} implies the desired result by \cref{eqn:complete-bipartite-three-var-i,eqn:complete-bipartite-three-var-ii}.

  \textbf{Case 1: $t\ge 0.999$, $x\ge 0.999$.}
  For simplicity of notation, write $a = 1-t$ and $b=1-x$.
  Then $0<a\le 0.001$ and $0\le b\le 0.001$.
  Let $c = 1-\left(t x + \frac{1-t}3\right) = \frac 23a +b -ab$.
  Write $g_n(u) = -(1-u)\log(1-u) - u \log \frac un \ge 0$.
  Then
  \begin{align}
    &~ G_3(1-a,1-b) \\
    \nonumber =&~
    \frac{\log 6}{\log 3} \cdot \left(\log 3 - g_2(c) \right) - (\log 2 - g_1(a) + t ( \log 3 - g_2(b))) \\
    \nonumber \le&~ \frac{\log 6}{\log 3} \cdot \left(\log 3 - g_2(c) \right) - (\log 2 - g_1(a) - a \log 3 + \log 3 - g_2(b)) \\
    \nonumber =&~ g_3(a) + g_2(b) - \frac{\log 6}{\log 3} \cdot g_2(c).
  \end{align}
  For $0\le w\le 0.001$, we have
  \begin{align}
    0.999 w \le -(1-w)\log(1-w) \le w.
  \end{align}
  So
  \begin{align}
    g_3(a) \le a \log \frac{3e} a, \qquad
    g_2(b) \le b \log \frac{2e} b, \qquad
    g_2(c) \ge c \left(0.999 + \log \frac 2c\right) =: h(c).
  \end{align}
  Because $h(c)$ is concave and increasing for $0\le c\le 0.002$, we have
  \begin{align}
    h(c) \ge \frac 12 h\left(\frac 43 a\right) + \frac 12 h(2b(1-a)) \ge \frac 12 h\left(\frac 43 a\right) + \frac 12 h(1.998 b).
  \end{align}
  For $0<a\le 0.001$ and $0\le b\le 0.001$, we have
  \begin{align}
    a \log \frac{3e} a < \frac 12 h\left(\frac 43 a\right), \qquad
    b \log \frac{2e} b \le \frac 12 h(1.998 b).
  \end{align}
  So $G_3(t,x) < 0$ for $0.999 \le t < 1$, $0.999\le x\le 1$.
  This finishes the proof for Case 1.

  \textbf{Case 2: $t\le 0.999$ or $x\le 0.999$.}
  Let $A = \left(\left[t^*,1\right]\times \left[\frac 13, 1\right]\right) \backslash \left([0.999,1] \times [0.999, 1]\right)$.
  Our goal is to prove that $G_3(t,x)<0$ for all $(t,x)\in A$.
  The proof strategy is as follows.
  We choose $\epsilon,\delta>0$ and a finite set $A^* \subseteq A$ such that the following are true.
  \begin{enumerate}[label=(\alph*)]
    \item \label{item:complete-bipartite:case-2:i} $G_3(t,x) < -\epsilon$ for all $(t,x)\in A^*$;
    \item \label{item:complete-bipartite:case-2:ii} For any $(t,x)\in A$, there exists $(t^*,x^*)\in A^*$ such that $\max\{|t-t^*|,|x-x^*|\}\le \delta$;
    \item \label{item:complete-bipartite:case-2:iii} For any $(t,x),(t',x')\in A$, if $\max\{|t-t'|,|x-x'|\}\le \delta$, then $|G_3(t,x)-G_3(t',x')|\le \epsilon$.
  \end{enumerate}
  If all three items hold, then they imply our goal.

  We take $\epsilon=0.00078$, $\delta=10^{-5}$, $A^* = A \cap \left(10^{-5} \bZ \times 10^{-5} \bZ\right)$.
  \cref{item:complete-bipartite:case-2:i} is verified using a computer program by iterating over all points in $A^*$.
  \cref{item:complete-bipartite:case-2:ii} is immediate by our choice of $A^*$.
  It remains to prove \cref{item:complete-bipartite:case-2:iii}.
  Note that on $A$, $f_3\left(t x+\frac{1-t}{3}\right)$ and $f_2(t) + t f_3(x)$ are non-decreasing in both $t$ and $x$.
  By convexity and monotonicity,
  \begin{align}
    \sup_{\substack{\frac 13 \le u,u'\le 1 \\ |u-u'|\le \delta}} |f_3(u)-f_3(u')| &= |f_3(1)-f_3(1-\delta)| \le 0.00014,\\
    \sup_{\substack{\frac 12\le u,u'\le 1 \\ |u-u'|\le \delta}} |f_2(u)-f_2(u')| &= |f_2(1)-f_2(1-\delta)| \le 0.00013.
  \end{align}
  So for $(t,x),(t',x')\in A$ with $\max\{|t-t^*|,|x-x^*|\}\le \delta$, we have
  \begin{align}
    |G_3(t,x)-G_3(t,x')| &\le
    \max\left\{ \frac{\log 6}{\log 3} \cdot \left| f_3\left(t x+\frac{1-t}{3}\right) -f_3\left(t x'+\frac{1-t}{3}\right) \right|,\right. \\
    \nonumber &\qquad \left.\left|(f_2(t) + t f_3(x))-(f_2(t) + t f_3(x'))\right|\right\} \\
    \nonumber &\le \max\left\{ \frac{\log 6}{\log 3} \cdot 0.00014, 0.00014\right\} \le 0.00023,\\
    |G_3(t,x')-G_3(t',x')| &\le
    \max\left\{ \frac{\log 6}{\log 3} \cdot \left| f_3\left(t x'+\frac{1-t}{3}\right) -f_3\left(t' x'+\frac{1-t'}{3}\right) \right|,\right. \\
    \nonumber &\qquad \left.\left|(f_2(t) + t f_3(x'))-(f_2(t') + t f_3(x'))\right|\right\} \\
    \nonumber &\le \max\left\{ \frac{\log 6}{\log 3} \cdot 0.00014, 0.00013 + \delta \log 3 \right\} \le 0.00023.
  \end{align}
  Therefore
  \begin{align}
    |G_3(t,x)-G_3(t',x')| &\le 0.00023 + 0.00023 = 0.00046 < \epsilon.
  \end{align}
  This proves \cref{item:complete-bipartite:case-2:iii}, thus finishing the proof for Case 2.

  \cref{eqn:complete-bipartite-two-var} follows by combining the three cases. This finishes the proof that $\eta_{\KL}(\pi,P)=\frac{\log 3}{2\log 6}$ and equality is achieved at and only at point distributions.
\end{proof}
Using the same proof strategy one could in principle prove the statement of \cref{prop:complete-bipartite} for any given $n\ge 3$ (assuming it is true). However it is unclear to us how to prove uniformly for all $n\ge 3$.


Finally, we provide sufficient conditions for the extremal distribution for $\eta_{\KL}$ to have full support. This result can be contrasted with \cref{ex:coloring,ex:complete-bipartite}, and \cite{bristiel2024entropy} where it is shown that the half-step entropy contraction for many natural chains (e.g., the random transposition model) has point measures as their only extremizers.
\begin{lemma} \label{lem:delta-extremal-full-supp}
  Let $(\pi, P)$ be a reversible pair where $P$ is irreducible.
  Suppose that either
  \begin{enumerate}[label=(\roman*)]
    \item \label{item:lem-delta-extremal-full-supp:i} $P(x,x)\ge \frac 12$ for all $x\in \cX$, and $\eta_{\KL}(\pi, P) > \frac 12$, or
    \item \label{item:lem-delta-extremal-full-supp:ii} $P(x,y) > 0$ for all $x,y\in \cX$.
  \end{enumerate}
  Let $\nu$ be a distribution on $\cX$ satisfying $\eta_{\KL}(\pi, P) = \frac{D(\nu P \| \pi P)}{D(\nu \| \pi)}$.
  Then $\nu$ has full support.
\end{lemma}
\begin{proof}

  Let $f = \frac{d\nu}{d\pi}$ be the Radon-Nikodym derivative.
  For the sake of contradiction assume that $\supp f \ne \cX$.
  Choose $a\in \cX - \supp f$ and $b\in \supp f$ such that $P(a,b) > 0$.
  Because $P$ is irreducible, such $(a,b)$ always exists.

  Let $h = \mathbbm{1}_a - \frac{\pi(a)}{\pi(b)}\mathbbm{1}_b$ and $f_\epsilon = f + \epsilon h$.
  For small enough $\epsilon>0$, we have $f_\epsilon\ge 0$ and $\pi[f_\epsilon]=1$.
  We prove that for $\epsilon > 0$ small enough, we have
  \begin{align} \label{eqn:lem-delta-extremal-full-supp:i}
  \frac{d}{d \epsilon} \frac{\Ent_\pi(P f_\epsilon)}{\Ent_\pi(f_\epsilon)} > 0.
  \end{align}

  Computation shows that
  \begin{align}
  \frac{d}{d \epsilon} \frac{\Ent_\pi(P f_\epsilon)}{\Ent_\pi(f_\epsilon)}
  = \frac 1{\Ent_\pi(f_\epsilon)^2} \left(\pi[Ph\log (P f_\epsilon)] \pi[f_\epsilon \log f_\epsilon] - \pi[h \log f_\epsilon] \pi[Pf_\epsilon \log (P f_\epsilon)]\right).
  \end{align}

  Expanding near $\epsilon=0$ gives
  \begin{align}
  \pi[h \log f_\epsilon] &= \pi(a) \log \epsilon + O(1),\\
  \pi[P h \log (P f_\epsilon)] &= \sum_{j\in \cX} \pi(j) \left(P(j,a)-P(j,b) \frac{\pi(a)}{\pi(b)}\right) \log (P f_\epsilon(j)) \\
  \nonumber &=\sum_{j\in \cX - \supp (P f)} \pi(j) P(j,a) \log (P(j,a) \epsilon) + O(1) \\
  \nonumber &=\pi(a) P(a, \cX - \supp(P f)) \log \epsilon + O(1),\\
  \pi[f_\epsilon \log f_\epsilon] &= \Ent_\pi(f) + o(1),\\
  \pi[P f_\epsilon \log (P f_\epsilon)] &= \Ent_\pi(P f) + o(1).
  \end{align}

  \textbf{Case \cref{item:lem-delta-extremal-full-supp:i}.}
  Because $P(a, b) > 0$, we have $a\in \supp(P f)$.
  Because $P$ is lazy, we have
  \begin{align}
    P(a, \cX - \supp(P f)) \le 1-P(a, a) \le \frac 12.
  \end{align}
  By assumption, $\frac{\Ent_\pi(P f)}{\Ent_\pi(f)} = \eta_{\KL}(\pi, P) > \frac12$.
  Therefore \cref{eqn:lem-delta-extremal-full-supp:i} holds for $\epsilon>0$ small enough.

  \textbf{Case \cref{item:lem-delta-extremal-full-supp:ii}.}
  In this case, $P(a, \cX - \supp(P f)) = 0$. So \cref{eqn:lem-delta-extremal-full-supp:i} holds for $\epsilon>0$ small enough.
\end{proof}

{
\section*{Acknowledgments}
The work of Y.G.~was supported in part by the National Science Foundation under Grant No.~DMS-1926686.
The work of Y.P.~was supported in part by the MIT-IBM Watson AI Lab and by the National Science Foundation under Grant No CCF-2131115.
}

\ifdefined\openprob
\section{Open problems}
\begin{enumerate}[label=(\arabic*)]
  \item characterize MLSI/SDPI for birth-death chains. We have LSI within constant factor (from
  Chen-SaloffCoste), but for MLSI only some lower bound under special constraints (Caputo etal).
  \item determine $\delta$ for random transposition model
  \item characterize $\cF_\pi$, the set of factorizable kernels wrt $\pi$, even for uniform $\pi$
\end{enumerate}
\fi

\bibliographystyle{alpha}
\bibliography{ref}
\end{document}